\documentclass[11pt]{article}

\title{Hitting time of edge disjoint Hamilton cycles in random subgraph processes on dense base graphs}

\author{
Yahav Alon
\thanks{School of Mathematical Sciences, Raymond and Beverly Sackler Faculty of Exact Sciences, Tel Aviv University,
Tel Aviv, 6997801, Israel. Email: yahavalo@mail.tau.ac.il.}
\and Michael Krivelevich
\thanks{School of Mathematical Sciences, Raymond and Beverly
Sackler Faculty of Exact Sciences, Tel Aviv University, Tel Aviv,
6997801, Israel. Email: krivelev@tauex.tau.ac.il. Partially supported by USA-Israel BSF grants 2014361 and 2018267, and by ISF grant 1261/17.}
}

\usepackage{tikz}
\usetikzlibrary{arrows,%
                petri,%
                topaths}%
\usepackage[position=top]{subfig}
\usepackage{amsmath,amsthm, amssymb,latexsym, amsfonts}
\usepackage{algorithm}
\usepackage{enumitem}
\usepackage[noend]{algpseudocode}

\begin{document}
\maketitle
\newtheorem{thm}{Theorem}
\newtheorem{propos}{Proposition}
\newtheorem{defin}{Definition}
\newtheorem{lemma}{Lemma}[section]
\newtheorem{corol}{Corollary}[section]
\newtheorem{thmtool}{Theorem}[section]
\newtheorem{corollary}[thmtool]{Corollary}
\newtheorem{lem}[thmtool]{Lemma}
\newtheorem{defi}[thmtool]{Definition}
\newtheorem{prop}[thmtool]{Proposition}
\newtheorem{clm}[thmtool]{Claim}
\newtheorem{conjecture}{Conjecture}
\newtheorem{problem}{Problem}
\newcommand{\Proof}{\noindent{\bf Proof.}\ \ }
\newcommand{\Remarks}{\noindent{\bf Remarks:}\ \ }
\newcommand{\Remark}{\noindent{\bf Remark:}\ \ }

\begin{abstract}
Consider the random subgraph process on a base graph $G$ on $n$ vertices: a sequence $\lbrace G_t \rbrace _{t=0} ^{|E(G)|}$ of random subgraphs of $G$ obtained by choosing an ordering of the edges of $G$ uniformly at random, and by sequentially adding edges to $G_0$, the empty graph on the vertex set of $G$, according to the chosen ordering. We show that if $G$ has one of the following properties:
\begin{enumerate}
	\item There is a positive constant $\varepsilon > 0$ such that $\delta (G) \geq \left( \frac{1}{2} + \varepsilon \right) n$;
	\item There are some constants $\alpha, \beta >0$ such that every two disjoint subsets $U,W$ of size at least $\alpha n$ have at least $\beta |U||W|$ edges between them, and the minimum degree of $G$ is at least $(2\alpha + \beta )\cdot n$;
\end{enumerate}
or:
\begin{enumerate}[resume]
	\item $G$ is an $(n,d,\lambda )$--graph, with $d\geq \frac{C\cdot n\cdot \log \log n}{\log n}$ and $\lambda \leq \frac{c\cdot d^2}{n}$ for some absolute constants $c,C>0$.
\end{enumerate}
then for a positive integer constant $k$ with high probability the hitting time of the property of containing $k$ edge disjoint Hamilton cycles is equal to the hitting time of having minimum degree at least $2k$. These results extend prior results by Johansson and by Frieze and Krivelevich, and answer a question posed by Frieze.
\end{abstract}

\section{Introduction} \label{sec-intro} 
Consider a \textit{random graph process}, defined as a random sequence of nested graphs on $n$ vertices $\tilde{G}(\sigma ) = \lbrace G_t(\sigma ) \rbrace _{t=0} ^{\binom{n}{2}}$, where $\sigma$ is an ordering of the edges of $K_n$ chosen randomly and uniformly from among all $\binom{n}{2}!$ such orderings. Set $G_0 (\sigma )$ to be a graph with vertex set $[n]$ and no edges, and for all $1\leq t \leq \binom{n}{2}$, $G_t (\sigma )$ is obtained by adding the $t$--th edge according to the order $\sigma$ to  $G_{t-1} (\sigma )$. The \textit{hitting time} of a monotone increasing, non--empty graph property $\mathcal{P}$, which we will denote as $\tau _{\mathcal{P}}( \tilde{G} (\sigma ))$, is a random variable equal to the index $t$ for which $G_t(\sigma ) \in \mathcal{P}$ but $G_{t-1}(\sigma ) \notin \mathcal{P}$.

Denote by $\mathcal{H}$ the property of Hamiltonicity, by $\mathcal{D}_d$ the property of having minimum degree at least $d$, and by $\tau _d(\tilde{G}(\sigma ))$ its hitting time in a random graph process $\tilde{G}(\sigma )$. A classical and very significant result by Ajtai, Koml{\'o}s and Szemer{\'e}di in \cite{AKS85}, and independently by Bollob{\'a}s in \cite{B84}, states that with high probability $\tau _2(\tilde{G}(\sigma )) = \tau _{\mathcal{H}}(\tilde{G}(\sigma ))$.

This result was generalized by Bollob{\'a}s and Frieze \cite{BF85}, in the following manner: let $\mathcal{A}_{k}$ be be the graph property of containing $\lfloor k/2 \rfloor$ edge disjoint Hamilton cycles, and a disjoint perfect matching in the case $k$ is odd. Then for $k$ constant, with high probability $\tau _k(\tilde{G}(\sigma )) = \tau _{\mathcal{A}_k}(\tilde{G}(\sigma ))$.

Some further results regarding the appearence of a number of edge disjoint Hamilton cycle were obtained over the years. Briggs, Frieze, Krivelevich, Loh and Sudakov \cite{BFKLS} proved an online version of the hitting time result: there is an algorithm that assigns a colour from the set $\{1,2,...,k\}$, online, to each edge added in the graph process, such that with high probability $G_{\tau _{2k}}$ contains $k$ Hamilton cycles $C_1, C_2, ... , C_k$, where the edges of cycle $C_j$ all have color $j$. Frieze and Krivelevich conjectured in \cite{FK} that for $0\leq p(n) \leq 1$ with high probability the random graph $G(n,p)$ contains $\lfloor \delta (G) /2 \rfloor$ edge disjoint Hamilton cycles. In the same paper, the authors prove this conjecture to be true for $p(n) = (1+o(1))\log n /n$. The conjecture was later proven to be true for all values of $p(n)$ over several subsequent papers (see \cite{FK05,BKS,KW,KKO,KKO2}).

\subsection*{The random subgraph process}
In this paper we aim to further generalize the result of Bollob{\'a}s and Frieze, regarding the hitting time of the property $\mathcal{A}_{k}$, by considering the \emph{random subgraph process} model.

\begin{defi}
Let $G$ be a graph on $n$ vertices and $m$ edges. For an ordering $\sigma$ of the set $E(G)$, the subgraph process $\lbrace G_t(\sigma ) \rbrace _{t=0} ^{m}$ on $G$ is a sequence of nested subgraphs of $G$ obtained by setting $G_0$ to be the empty graph on $n$ vertices, and $G_t$ to be the result of adding the $t$'th edge (according to $\sigma$) to $G_{t-1}$, for $t>0$.

\emph{The random subgraph process} $\lbrace G_t \rbrace _{t=0} ^{m}$ \emph{on base graph} $G$ is the (random) graph process on $G$ obtained by choosing an edge ordering $\sigma$ at random, uniformly from among all $m!$ possible orderings.
\end{defi}

The \emph{hitting time} of a monotone increasing graph property $\mathcal{P}$ such that $G\in \mathcal{P}$ in a random subgraph process, defined exactly the same as in the random graph process model, can now be considered for different properties and base graphs. In his recent paper \cite{JO}, Johansson proved the following result regarding the hitting time of the properties $\mathcal{H}$ and $\mathcal{D}_2$, under an assumption on the minimum degree of the base graph:\\

\noindent \textbf{Theorem.} (Johansson, \cite{JO}): \emph{Let $\varepsilon>0$, let $G$ be a graph such that $\delta (G) \geq \left( \frac{1}{2} + \varepsilon \right) n$ and let $\lbrace G_t \rbrace $ be a random subgraph process with base graph $G$. Then with high probability $\tau _2 \left( \lbrace G_t \rbrace \right) = \tau _{\mathcal{H}} \left( \lbrace G_t \rbrace \right)$.}\\

From this result, the author further derived a threshold probability for Hamiltonicity in the random subgraph model $G_p$, for base graphs $G$ with minimum degree at least $\left( \frac{1}{2} + \varepsilon \right) n$.\\

In this paper we present a generalization of this result in two directions:
\begin{enumerate}
	\item We consider the hitting times property $\mathcal{A}_{2k}$ for every constant positive integer $k$;
	\item We extend the result to a larger class of base graphs.
\end{enumerate}
This generalizes Johansson's result, as well as the result by Bollob{\'a}s and Frieze, and answers a question by Frieze (see \cite{FRI}, Problem 20).

\subsection*{Dirac graphs and super--Dirac graphs}
We say that a graph $G$ on $n$ vertices is a \emph{Dirac graph} if its minimum degree is at least $\frac{1}{2}n$. The classical Dirac's theorem \cite{DIR} states that if $G$ is a Dirac graph on $n\geq 3$ vertices then it is Hamiltonian. For some $\varepsilon > 0$ we say that a graph is $\varepsilon$--super--Dirac if its minimum degree is at least $\left( \frac{1}{2} + \varepsilon \right) n$. From Dirac's theorem it is easily derived that an $\varepsilon$--super--Dirac graph contains a set of edge disjoint Hamilton cycles of size linear in $n$.

The study of Dirac and super--Dirac graphs has yielded some interesting results with regards to Hamiltonicity in random subgraphs. Krivelevich, Lee and Sudakov \cite{KLS} showed that there is some constant $C$ such that if $G$ is a Dirac graph and $p\geq C\log n/n$ then the random subgraph $G_p$ is with high probability Hamiltonian. Frieze and Johansson \cite{FJ} showed that if $G$ is $\varepsilon$--super--Dirac and $k$ is large enough as a function of $\varepsilon$, then the random subgraph $G(k\mbox{-out})$ is with high probability Hamiltonian.

Johansson's result states that the hitting time of Hamiltonicity in a random subgraph process is with high probability equal to the hitting time of having minimum degree at least 2, when the base graph is $\varepsilon$--super--Dirac. Our first result in this paper, a proof of which is presented in Section \ref{sec-superdirac}, is an extension of Johansson's result on $\varepsilon$--super--Dirac graphs to the hitting time of the property $\mathcal{A}_{2k}$, for a constant $k$:

\begin{thm}  \label{thmDirac}
Let $\varepsilon > 0$, $k\in \mathbb{N}$, let $G$ be an $\varepsilon$--super--Dirac graph and let $\lbrace G_t \rbrace $ be a random graph process with base graph $G$. Then with high probability $\tau _{2k}\left( \lbrace G_t \rbrace \right) = \tau _{\mathcal{A}_{2k}}\left( \lbrace G_t \rbrace \right)$.
\end{thm}

\subsection*{Pseudo--random graphs}
In Section \ref{sec-abdense} and Section \ref{sec-ndl} we extend this hitting time result to additional families of graphs. In both cases, the family of graphs for which we show that the hitting times are with high probability equal is in some way connected to the notion of pseudo--random graphs. Informally, a pseudo--random graph is a graph that has some of the characteristics one expects to observe in a random graph (for some general information on pseudo--random graphs, the reader may refer to \cite{KSSURV}). Thomason suggested the following definition:

\begin{defi}[Thomason, 1987 \cite{THOM1,THOM2}]\label{def-prg}
A graph $G$ on $n$ vertices is said to be $(p,\alpha )$--\emph{jumbled}, with $0 < p < 1 \leq \alpha$, if for every subset $U\subseteq V(G)$ the inequality $$ \left| e(U) - p\cdot \binom{|U|}{2} \right| \leq \alpha |U| $$ holds.
\end{defi}

In Section \ref{sec-abdense} we present a result on a similar family of graphs. We introduce a definition of $(\alpha ,\beta )$\emph{--dense} graphs, which omits the upper bound requirement on the number of edges spanned by two vertex subsets, while adding a minimum degree requirement:

\begin{defi}\label{def-abdense}
Let $0 < \alpha , \beta \leq 1$. We say that a graph $G$ on $n$ vertices is $(\alpha ,\beta )$\emph{--dense} if
\begin{enumerate}
	\item $\forall A,B\subseteq V(G)$ disjoint subsets such that $|A|,|B| \geq \alpha n$ : $e_G(A,B) \geq \beta \cdot |A|\cdot |B|$;
	\item $\delta (G) \geq \left( 2\alpha + \beta \right) n$.
\end{enumerate}
\end{defi}

Thomason \cite{THOM1} showed that if $G$ is a $(p,\alpha )$--jumbled graph, and $\delta (G) = \Omega \left( \alpha / p \right)$, then $G$ is Hamiltonian. This property extends to the similarly defined $(\alpha ,\beta )$--dense graphs: by a simple rotation and extension argument, it is easy to derive that for all constant values of $0 < \alpha , \beta \leq 1$ if $G$ is $(\alpha ,\beta )$--dense then it is Hamiltonian, and in fact contains a set of edge disjoint Hamilton cycles of size linear in $n$.

In Section \ref{sec-abdense} we prove the following theorem:

\begin{thm} \label{thmJumbled}
Let $0 < \alpha , \beta \leq 1$, $k\in \mathbb{N}$, let $G$ be an $(\alpha ,\beta )$--dense graph on $n$ vertices and let $\lbrace G_t \rbrace $ be a random graph process with base graph $G$. Then with high probability $\tau _{2k}\left( \lbrace G_t \rbrace \right) = \tau _{\mathcal{A}_{2k}}\left( \lbrace G_t \rbrace \right)$.
\end{thm}

A family of pseudo--random graphs of particular interest is given by the following definition:

\begin{defi}
Let $\lambda >0$ and let $d$ be a positive integer. A graph $G$ on $n$ vertices is called an $(n,d,\lambda )$--graph if $G$ is a $d$-regular graph with the second largest eigenvalue of its adjacency matrix in absolute value equal to $\lambda$.
\end{defi}

The following lemma due to Alon and Chung provides a connection between this definition, based on graph eigenvalues, and Thomason's definition of a pseudo--random/jumbled graph:

\begin{lemma}[Expander mixing lemma \cite{MIX}] \label{lemmaMixing}
Let $G$ be an $(n,d,\lambda )$--graph. Then for every pair of disjoint vertex subsets $U,W\subseteq V(G)$ it holds that $$\left|e_G(U,W) - \frac{d}{n}|U|\cdot |W| \right| \leq \lambda \cdot \sqrt{|U|\cdot |W|} .$$
\end{lemma}

Frieze and Krivelevich \cite{FK05} showed that if $d=\Theta (n)$ and $\lambda = o\left( d \right)$ then an $(n,d,\lambda )$--graph contains a set of edge disjoint Hamilton cycles of size $(1-o(1))\frac{d}{2}$. The same authors also showed \cite{FK2} that if $G$ is an $(n,d,\lambda )$--graph with $\lambda = o\left( \frac{d^{5/2}}{(n\log n)^{3/2}} \right)$ then the hitting time of Hamiltonicity in a random subgraph process on $G$ is with high probability equal to that of having minimum degree at least 2.

In Section \ref{sec-ndl} we prove the following theorem:

\begin{thm} \label{thmNDL}
Let $C=10^8$,$c=1/400$, $d=d(n)\geq C\cdot \frac{n\cdot \log \log n}{\log n}$, $\lambda = \lambda (n) \leq \frac{c\cdot d^2}{n}$, $k\in \mathbb{N}$, let $G$ be an $(n,d,\lambda )$--graph, and let $\lbrace G_t \rbrace $ be a random graph process with base graph $G$. Then with high probability $\tau _{2k}\left( \lbrace G_t \rbrace \right) = \tau _{\mathcal{A}_{2k}}\left( \lbrace G_t \rbrace \right)$.
\end{thm}

This extends Frieze and Krivelevich's result in the dense regime by softening the restriction $\lambda = o\left( \frac{d^{5/2}}{(n\log n)^{3/2}} \right)$ to just $\lambda \leq  \frac{cd^2}{n}$, as well as generalizing the Hamiltonicity property to $\mathcal{A}_{2k}$.

\subsection*{Paper structure}
In Section \ref{sec-per} we present some preliminaries. In Sections \ref{sec-superdirac} through \ref{sec-ndl} we give proofs of our results. In Section \ref{sec-remarks} we discuss the tightness of our results and some open questions.

\section{Preliminaries} \label{sec-per}
In this section we gather several definitions and results to be used in the following sections.

Throughout the paper, it is assumed that all logarithms are in the natural base, unless explicitly stated otherwise. We suppress the rounding notation occasionally to simplify the presentation.

The following standard graph theoretic notations will be used:
\begin{itemize}
\item $N_G(U)$ : the external neighbourhood of a vertex subset $U$ in the graph $G$, i.e.
$$
N_G(U) = \lbrace v \in V(G)\setminus U:\ v\ \mbox{has\ a\ neighbour\ in}\ U \rbrace.
$$
\item $e_G(U)$: the number of edges spanned by a vertex subset $U$ in a graph $G$. This will sometimes be abbreviated as $e(U)$, when the identity of $G$ is clear from the context.
\item $e_G(U,W)$: the number of edges of $G$ between the two disjoint vertex sets $U,W$. This will sometimes be abbreviated as $e(U,W)$ when $G$ is clear from the context.
\item $\nu _{G} (U, W)$: the maximum size of a matching in $G$ between the two disjoint vertex sets $U,W$.
\item $p_P(v),s_P(v)$: respectively, the predecessor and the successor of a vertex $v$ along the (oriented) path $P$. We will write $p(v),s(v)$ when $P$ is clear from the context.
\end{itemize}

With regards to the hitting time of a graph property $\mathcal{P}$ in some graph process, we abbreviate the notation to $\tau _{\mathcal{P}}$, assuming that the discussed graph process is clear from context.

\subsection*{Probabilistic and combinatorial bounds}

\begin{lemma}\label{binomial-coeff-bounds}
Let $1\leq \ell \leq k \leq n$ be integers. Then the following inequalities hold:
\begin{itemize}
\item $\binom{n}{k} \leq \left( \frac{en}{k} \right) ^k$\,;
\item $\frac{\binom{n-\ell}{k}}{\binom{n}{k}} \leq e^{-\frac{\ell \cdot k}{n}}$\,.
\end{itemize}
\end{lemma}

\begin{lemma}\label{chernoff}
\emph{(Chernoff bound for binomial tails, see e.g. \cite{CHER})} Let $X\sim Bin(n,p)$. Then the following inequalities hold:
\begin{itemize}
\item $Pr(X <(1-\delta )np) \leq \exp \left( -\frac{\delta ^2 np}{2} \right)$ for every $\delta > 0$\,;
\item $Pr(X >(1+\delta )np) \leq \exp \left( -\frac{\delta ^2 np}{3} \right)$ for every $0 < \delta < 1$\,;
\item $Pr(X >(1+\delta )np) \leq \exp \left( -\frac{\delta np}{3} \right)$ for every $\delta \geq 1$\,.
\end{itemize}
\end{lemma}

\subsection*{Graph theory}

In the paper we will use of the following definition of an expanding graph:

\begin{defi}
Let $G$ be a graph, and let $\alpha >0$ and $k$ be a positive integer. The graph $G$ is a $(k,\alpha )$--expander if for every vertex subset $U\subseteq V(G)$ with $|U| \leq k$ it holds that $|N_G(U)| \geq \alpha |U|$.
\end{defi}

\subsection*{Asymptotic equivalence of random models}

We use the following standard definitions of random subgraph models: $G_p$ is the space of random subgraphs of the graph $G$, obtained by keeping each edge in $E(G)$ with probability $p$, independently of all other edges. $G_m$ is the space of random subgraphs of $G$ obtained by randomly choosing exactly $m$ of the edges in $E(G)$ to keep, uniformly from among the $\binom{|E(G)|}{m}$ possibilities.

Throughout the paper we will often find it more convenient to bound the probability of some properties holding in $G_p$, while the nature of the problems this paper discusses requires us to provide bounds in the model $G_m$. The following lemma provides us with a useful connection between the two models:

\begin{lemma}[see e.g. \cite{RGBOOK} chapter 1.4]\label{lem-asym}
Let $G$ be a graph on $n$ vertices and $m$ edges, $\mathcal{P}$ a monotone increasing graph property, $0\leq t \leq m$ and $p = t / m$. Then
$$ Pr \left[ G_t \notin \mathcal{P} \right] \leq 3\sqrt{m} \cdot Pr \left[ G_p \notin \mathcal{P} \right] .$$
Furthermore, if $\lim _{n\rightarrow \infty} Pr \left[ G_p \notin \mathcal{P} \right] =0$ then $\lim _{n\rightarrow \infty} Pr \left[ G_t \notin \mathcal{P} \right] =0$, and if $\lim _{n\rightarrow \infty} Pr \left[ G_p \notin \mathcal{P} \right] =1$ then $\lim _{n\rightarrow \infty} Pr \left[ G_t \notin \mathcal{P} \right] =1$.
\end{lemma}

\subsection*{Rotations, boosters and booster pairs}

In our proofs we will strongly rely on P{\'o}sa's notion of rotations. In short, a rotation is a way in which one may obtain several paths on the same vertex set in a graph. More formally, let $P = (v_1,...,v_k)$ be a path in a graph $G$. Suppose that $(v_i,v_k)\in E(G)$ for some $1\leq i <k-1$. Then a new path $P'$ on the same set of vertices and with the same length and starting point can be obtained by removing the edge $(v_i,v_{i+1})$ from $P$ and replacing it with $(v_i,v_k)$. This operation is called a rotation.

By performing sequences of rotations, we can hope to get many alternative endpoints for longest paths in a graph. Under certain conditions this is indeed possible, as expressed by the following lemma:

\begin{lemma}[P{\'o}sa '76 \cite{POS}]
Let $k$ be an integer, $G$ be a graph on $n$ vertices such that $G$ is a $(k,2)$-expander, and let $P$ be a longest path in $G$. Then there are at least $k+1$ end vertices of paths in $G$ obtained from $P$ be sequences of rotations with  the starting point fixed.
\end{lemma}

We will also use the notion of a \emph{booster pair}, or a \emph{BP} in short, which was introduced by Montgomery \cite{MONT}.

\begin{defi}
Let $G$ be a graph. A pair $\lbrace e_1,e_2 \rbrace$ of non-edges $e_1,e_2\in \binom{V(G)}{2} \setminus E(G)$ is called a \emph{booster pair}, or a \emph{BP}, if the graph $G^{\prime}$ with edge set $E(G^{\prime})=E(G)\cup \lbrace e_1,e_2 \rbrace$ is either Hamiltonian or has a path longer than a longest path of $G$.
\end{defi}

\section{$\varepsilon$--super--Dirac graphs} \label{sec-superdirac}

In this section we provide a proof of Theorem \ref{thmDirac}. Throughout the section, $G$ is some fixed graph assumed to be $\varepsilon$--super--Dirac.

In order to prove the theorem, we will prove this sufficient condition: with high probability for every $F\subseteq G_{\tau _{2k}}$ with $\Delta (F) \leq 2k-2$, the graph $G_{\tau _{2k}} \setminus F$ is Hamiltonian. We aim to prove this by proving two main lemmas:

\begin{lemma} \label{lemma1}
With high probability for every $F\subseteq G_{\tau _{2k}}$ with $\Delta (F) \leq 2k-2$, $G_{\tau _{2k}} \setminus F$ contains a subgraph $\Gamma _1$ which is a connected $\left( \frac{n}{8},2\right)$--expander with at most $5\varepsilon ^2 n\log n +1$ edges.
\end{lemma}

\begin{lemma} \label{lemma2}
With high probability for every $F\subseteq G_{\tau _{2k}}$ with $\Delta (F) \leq 2k-2$ and for every subgraph $\Gamma $ of $G_{\tau _{2k}} \setminus F$ which is a connected, non--Hamiltonian $\left( \frac{n}{8},2\right)$--expander with at most $6\varepsilon ^2 n\log n$ edges, $G_{\tau _{2k}}\setminus F$ contains a booster pair with respect to $\Gamma$. 
\end{lemma}

Putting the two lemmas together we easily obtain the theorem:

Say we have found $h$ edge disjoint Hamilton cycles in $G_{\tau _{2k}}$, with $0\leq h \leq k-1$, and denote their union as $F$. By Lemma \ref{lemma1}, with high probability $G_{\tau _{2k}} \setminus F$ contains an $\left( \frac{n}{8},2\right)$--expanding subgraph $\Gamma _1$ with at most $5\varepsilon ^2 n\log n +1$ edges. If $\Gamma _1$ is not Hamiltonian, by Lemma \ref{lemma2}, $G_{\tau _{2k}}\setminus F$ contains a \emph{BP} with respect to $\Gamma _1$. By adding this \emph{BP} to $\Gamma _1$ we obtain a new sparse expander $\Gamma _2$, which is either Hamiltonian or has a longest path which is strictly longer than that of $\Gamma _1$. By continuing to apply Lemma \ref{lemma2} at most $n$ times, we obtain the theorem.

Throughout the proofs we will assume that $\varepsilon$ is small enough, without explicitly stating so.

\subsection{Proof of Lemma \ref{lemma1}}

\noindent For a graph $\Gamma$ on $n$ vertices, let $$d_0:=5\varepsilon ^2 \log n,$$ and $$SMALL(\Gamma ):=\lbrace v\in V(\Gamma ): d_{\Gamma}(v)\leq d_0 \rbrace.$$ We will first show that WHP $G_{\tau _{2k}}$ has the following properties:
\begin{itemize}
	\item[(P1)] $\Delta \left( G_{\tau _{2k}} \right) \leq 10\log n$;
	\item[(P2)] $\left| SMALL \left(G_{\tau _{2k}} \right) \right| \leq n^{\frac{1}{2}(1-\varepsilon)}$;
	\item[(P3)] $\forall u,v\in SMALL \left(G_{\tau _{2k}} \right) : dist_{G_{\tau _{2k}}}(u,v) > 4$;
	\item[(P4)] $\forall U\subset V(G)\ s.t.\ 4\varepsilon ^2\log n\leq |U| \leq \frac{5n}{\sqrt{\log n}}:e_{G_{\tau _{2k}}}(U)<0.1\cdot \varepsilon ^2|U|\log n$;
	\item[(P5)] $\forall U,W\subset V(G)\ disjoint\ s.t.\ \frac{n}{\sqrt{\log n}} \leq |U| \leq \frac{n}{8},|W|=2.5|U|:e_{G_{\tau _{2k}}}(U,V\setminus (U\cup W)) \geq 0.04|U|\log n$;
	\item[(P6)] $\forall U\subset V(G)\ s.t.\ \frac{n}{3} \leq |U|\leq \left( \frac{1}{2} + \frac{\varepsilon}{2} \right) n: e_{G_{\tau _{2k}} }\left( U, V(G)\setminus U \right) > n\sqrt{\log n}$.
\end{itemize}

\begin{proof} For each property, we will bound the probability of $G_{\tau _{2k}}$ failing to have it separately.\\
Let $N = |E(G)|$, and let
$$p_1 = \frac{\log n}{n},\ p_2 = \frac{2\log n}{n},\ m_1 = N\cdot p_1,\ m_2 = N\cdot p_2.$$

\begin{lemma} \label{lemmaTAU}
With high probability $m_1 \leq \tau _{2k} \leq m_2$.
\end{lemma}
\begin{proof}
Since having minimum degree at least $2k$ is a monotone increasing property, by Lemma \ref{lem-asym} it suffices to show that
\begin{enumerate}
	\item $Pr\left[ \delta \left( G_{p_1} \right) \geq 2k \right] = o(1)$;
	\item $Pr\left[ \delta \left( G_{p_2} \right) < 2k \right] = o(1)$.
\end{enumerate}
We will prove both bounds.

\begin{enumerate}
	\item This bound is already well known, as $Pr\left[ \delta \left( G_{p_1} \right) \geq 2k \right] \leq Pr\left[ \delta \left( G(n,p_1) \right) \geq 2k \right] = o(1)$.
	\item Let $v\in V(G)$. The probability that $d_{G_{p_2}}(v) <2k$ is at most
	\begin{eqnarray*}
	Pr\left[ d_{G_{p_2}}(v) <2k \right] & \leq & \sum\limits_{i=0}^{2k-1} \binom{d_G(v)}{i}{p_2}^i(1-p_2)^{d_G(v)-i} \\
	& \leq & \sum\limits_{i=0}^{2k-1} \binom{n-1}{i}{p_2}^i(1-p_2)^{\left( \frac{1}{2} + \varepsilon \right) n -i} \\
	& \leq & \sum\limits_{i=0}^{2k-1} (6\log n)^i \exp \left( - \left(1 + 2\varepsilon \right) \log n \right) \\
	& \leq & n^{-1-\varepsilon} .
	\end{eqnarray*}
	By the union bound we get
	$$Pr\left[ \delta \left( G_{p_2} \right) < 2k \right] \leq \sum _{v\in V(G)} Pr\left[ d_{G_{p_2}}(v) <2k \right] \leq n^{-\varepsilon} = o(1).$$
\end{enumerate}

\end{proof}

\noindent Applying Lemma \ref{lem-asym}, given the conclusion of Lemma \ref{lemmaTAU}, we can argue that for a monotone increasing property $(P)$ it is sufficient to prove that $Pr\left( G_{p_1} \notin (P) \right) = o(1),$ and that for a monotone decreasing property it is sufficient to prove that $Pr\left( G_{p_2} \notin (P) \right) = o(1)$.

\begin{itemize}
	\item[(P1)]
	
	\begin{eqnarray*}
		Pr\left[ G_{p_2} \notin \bf{(P1)} \right]  & \leq & n\cdot Pr\left[ Bin\left( n,p_2\right) \geq 10\log n \right] \leq n\cdot \exp \left( -\frac{2\cdot 4\log n}{3} \right) = o(1).
	\end{eqnarray*}
	
	\item[(P2)] The probability that $G_{p_1} \notin (P2)$ is at most the probability that there is some subset $U\subseteq V(G)$ of size $n^{\frac{1}{2}(1-\varepsilon)}$ such that $e_{G_{p_1}}(U,V(G)\setminus U) \leq n^{\frac{1}{2}(1-\varepsilon)} \cdot d_0$. Since $|U| = o(n)$, there are at least $\left( \frac{1}{2}+\frac{9\varepsilon}{10} \right) |U|\cdot n$ edges of $G $ between $U$ and $V(G)\setminus U$. So
	\begin{eqnarray*}
		Pr\left[ G_{p_1} \notin \bf{(P2)} \right] & \leq & \binom{n}{n^{\frac{1}{2}(1-\varepsilon)}} \cdot Pr\left[ Bin\left( \left( \frac{1}{2}+\frac{9\varepsilon}{10} \right) n^{1.5-\frac{1}{2}\varepsilon},p_1 \right) \leq n^{\frac{1}{2}(1-\varepsilon)} \cdot d_0 \right] \\
		& \leq & \left( en^{\frac{1}{2}(1+\varepsilon)} \right) ^{n^{\frac{1}{2}(1-\varepsilon)}} \cdot \sum\limits_{i=0}^{n^{\frac{1}{2}(1-\varepsilon)} \cdot d_0} \binom{n^{1.5-\frac{1}{2}\varepsilon}}{i} \cdot {p_1}^i \cdot (1-p_1)^{\left( \frac{1}{2}+\frac{9\varepsilon}{10} \right) n^{1.5-\frac{1}{2}\varepsilon} - i} \\
		& \leq & \left( en^{\frac{1}{2}(1+\varepsilon)} \right) ^{n^{\frac{1}{2}(1-\varepsilon)}} \cdot (n^{\frac{1}{2}(1-\varepsilon)} \cdot d_0 +1) \binom{n^{1.5-\frac{1}{2}\varepsilon}}{n^{\frac{1}{2}(1-\varepsilon)} \cdot d_0} \cdot {p_1}^{n^{\frac{1}{2}(1-\varepsilon)} \cdot d_0} \\
		& & \cdot (1-p_1)^{\left( \frac{1}{2}+\frac{9\varepsilon}{10} \right) n^{1.5-\frac{1}{2}\varepsilon} - n^{\frac{1}{2}(1-\varepsilon)} \cdot d_0} \\
		& \leq & o(n)\cdot \left( en^{\frac{1}{2}(1+\varepsilon)} \cdot \left( \frac{en^{1.5-\frac{1}{2}\varepsilon}\cdot p_1}{n^{\frac{1}{2}(1-\varepsilon)} \cdot d_0 \cdot (1-p_1)} \right) ^{d_0 } \cdot \exp \left({-\left( \frac{1}{2}+\frac{4\varepsilon}{5} \right)np_1}\right)  \right) ^{n^{\frac{1}{2}(1-\varepsilon)}} \\
		& \leq & o(n)\cdot \left( en^{\frac{1}{2}(1+\varepsilon)} \cdot \left( \frac{e}{4\varepsilon ^2} \right) ^{5\varepsilon ^2 \log n } \cdot \exp \left({-\left( \frac{1}{2}+\frac{4\varepsilon}{5} \right)np_1}\right)  \right) ^{n^{\frac{1}{2}(1-\varepsilon)}} \\
		& \leq & n^{-\frac{1}{4}\varepsilon \cdot n^{\frac{1}{2}(1-\varepsilon)} } = o(1).
	\end{eqnarray*}
	
	\item[(P3)] As $m_1 \leq \tau _{2k} \leq m_2$ with high probability, it is sufficient to prove that with high probability $G_{m_2}$ does not contain a path of length at most 4 between two (not necessarily distinct) members of $SMALL\left( G_{m_1} \right)$.\\ 
	We will prove this by showing that:
	\begin{enumerate}
		\item With high probability $G_{m_1}$ does not contain a path of length at most 4 between two (not necessarily distinct) members of $SMALL\left( G_{m_1} \right)$;
		\item With high probability, adding $m_2-m_1 = m_1$ random edges of $G\setminus G_{m_1}$ to $G_{m_1}$ does not result in a path between two members of $SMALL\left( G_{m_1} \right)$.
	\end{enumerate}
	
	First, we bound the probability that a specific path $P$ of length $\ell$ is in $G_{m_1}$:
	$$
	Pr \left[ P \in G_{m_1} \right] = \frac{ \binom{N-\ell}{m_1-\ell} }{ \binom{N}{m_1} } \leq \left( \frac{m_1}{N} \right) ^{\ell} = \left( \frac{\log n}{n} \right) ^{\ell}.
	$$
	Next, we bound the probability that the two endpoints of a path $P$, denoted as $s,t$, are members of $SMALL\left( G_{m_1} \right)$, conditioned on the event $P \in G_{m_1}$. This probability is at most the probability that the vertex set $\{ s,t\}$ has at most $d_1 := 2\cdot d_0 -2$ edges between it and $V(G)\setminus \{s,t\}$ in $G_{m_1}$, not including the two edges belonging to $P$. Since there are at least $(1+\varepsilon )n$ edges of $G$ between $\{s,t\}$ and there rest of the graph, not including the two edges of $P$, we get
	
	\begin{eqnarray*}
		Pr\left[ s,t\in SMALL\left( G_{m_1}  \right) | P\in G_{m_1} \right] & \leq & \sum\limits_{i=0}^{d_1} \binom{2n-4}{i}\cdot \frac{ \binom{N-\ell -(1+\varepsilon )n}{m_1 -\ell -i} }{\binom{N-\ell}{m_1-\ell}} \\
		& \leq & (d_1+1)\binom{2n}{d_1}\cdot \frac{ \binom{N-\ell -(1+\varepsilon )n}{m_1 -\ell -d_1} }{\binom{N-\ell}{m_1-\ell}} \\
		& \leq & d_1\left( \frac{2en}{d_1} \right)^{d_1}\cdot \left( \frac{m_1-\ell}{N-\ell} \right)^{d_1}\\
		& & \cdot \exp \left( -\frac{(m_1-\ell -d_1)((1+\varepsilon) n-d_1)}{N-\ell -d_1} \right) \\
		& \leq & \left( 10\varepsilon ^2 \right)^{-10\varepsilon ^2 \log n} \exp (-(1+0.9\varepsilon ) \log n) \\
		& \leq & n^{-1-0.8\varepsilon}.
	\end{eqnarray*}
	
	Here, in the third inequality we used the fact that $\binom{N-\ell}{m_1-\ell} \ge \left( \frac{N-\ell}{m_1-\ell} \right) ^{d_1} \cdot \binom{N-\ell -d_1}{m_1-\ell -d_1}$, and Lemma \ref{binomial-coeff-bounds}.
	
	Apply the union bound to bound the probability that there is a path $P\in G_{m_1}$ of length $1\leq \ell \leq 4$, such that both of its endpoints are in $SMALL\left( G_{m_1} \right)$:
	
	$$
	Pr\left[ \exists P\in G_{m_1} :\ s,t\in SMALL\left( G_{m_1} \right) \right] \leq \sum_{\ell =1}^4 n^{\ell +1}\cdot \left( \frac{\log n}{n} \right) ^{\ell} \cdot n^{-1-0.8\varepsilon} = o(1).
	$$
	
	Finally, we bound the probability that adding $m_2 - m_1$ random edges to $G_{m_1}$ results in a path between two vertices of $SMALL\left( G_{m_1} \right)$. Since we already proved that with high probability $G_{m_1} \in (P2)$ and $G_{m_2} \in (P1)$, we can assume for the sake of calculation that indeed $G_{m_1},G_{m_2}$ have these properties.\\
	For the set of $m_2-m_1$ added edges to close a path, at least one of the edges must have both its vertices inside the set of vertices at distance at most 3 from $SMALL\left( G_{m_1} \right)$ in $G_{m_2}$ . By $(P1),(P2)$, this set is of size at most $n^{\frac{1}{2}(1-\varepsilon )}\cdot (10\log n)^3 \leq n^{\frac{1}{2}-0.4\varepsilon}$. By the union bound, the probability that at least one of the added edges is in this set is at most
	$$
	(m_2-m_1)\cdot \frac{ \binom{n^{\frac{1}{2}-0.4\varepsilon}}{2} }{N-m_2} = o(1).
	$$
	
	\item[(P4)]
	\begin{eqnarray*}
		Pr\left[ G_{p_2} \notin {\bf (P4)} \right] & \leq & \sum \limits _{i=4\varepsilon ^2\log n}^{\frac{5n}{\sqrt{\log n}}} \binom{n}{i} Pr\left[ Bin\left( \binom{i}{2}, p_2 \right) \geq 0.1\varepsilon ^2i\log n \right] \\
		& \leq & \sum \limits _{i=4\varepsilon ^2\log n}^{\frac{5n}{\sqrt{\log n}}} \left( \frac{en}{i} \right) ^i \cdot \left( \frac{5ep_2\cdot i^2}{\varepsilon ^2 i\log n} \right) ^{0.1\varepsilon ^2 i\log n} \\
		& \leq & \sum \limits _{i=4\varepsilon ^2\log n}^{\frac{5n}{\sqrt{\log n}}} \left( \frac{en}{i} \right) ^i \cdot \left( \frac{10e\cdot i}{\varepsilon ^2 n} \right) ^{0.1\varepsilon ^2 i\log n} \\
		& \leq & \exp \left( {-\omega (\log n) } \right) .
	\end{eqnarray*}

	\item[(P5)] We will use the fact that if $|U| = i, |W|=2.5i$ and $\frac{n}{\sqrt{\log n}} \leq i \leq \frac{n}{8}$ then, since $\delta (G) \ge \left(\frac{1}{2} + \varepsilon \right) n$, we have
	\begin{eqnarray*}
	e_G\left( U,V(G)\setminus (U \cup W) \right) & \geq & \delta (G)\cdot |U| -2e_G(U)-e_G(U,W) \\
	& \ge & i\cdot \left( \frac{1}{2} + \varepsilon \right) n -3.5i^2\\
	& \geq & \frac{1}{20}i\cdot n.
	\end{eqnarray*}
	
	Therefore by the union bound
	\begin{eqnarray*}
		Pr\left[ G_{p_1} \notin {\bf (P5)} \right] & \leq & \sum\limits_{i=\frac{n}{\sqrt{\log n}}}^{ n/8} \binom{n}{i}\binom{n}{2.5i} Pr\left[ Bin\left( 0.05 in , p_1\right) <0.04i\log n \right] \\
		& \leq & 3^n\cdot \exp \left( -\Omega \left( n^2p_1/\sqrt{\log n} \right) \right) = o(1) .
	\end{eqnarray*}
	
	\item[(P6)] We will use the fact that if $|U| = i \leq \left( \frac{1}{2} + \frac{\varepsilon}{2} \right) n$ then $e_G\left( U,V(G)\setminus U \right)\geq i\cdot \frac{\varepsilon}{2}\cdot n$. By the union bound
	\begin{eqnarray*}
		Pr\left[ G_{p_1} \notin {\bf (P6)} \right] & \leq & \sum\limits_{i=\frac{n}{3}}^{\left( \frac{1}{2} + \frac{\varepsilon}{2}\right) n} \binom{n}{i} Pr\left[ Bin\left( \varepsilon in/2 , p_1 \right) \leq n\sqrt{\log n} \right] \\
		& \leq & 2^n\cdot \exp \left( - \Omega \left( n \log n \right) \right) = o(1) .
	\end{eqnarray*}
	
\end{itemize}

\end{proof}

\begin{lemma} \label{lemma1.1}
	Let $\Gamma$ be a graph on $n$ vertices such that $\delta \left( \Gamma \right) \geq 2k$ and $\Gamma \in (P1)$--$(P5)$, and let $F\subseteq \Gamma$ with $\Delta (F) \leq 2k-2$. Then $\Gamma \setminus F$ contains a subgraph $\Gamma _0$ which is an $\left( \frac{n}{8},2 \right)$--expander with at most $5\varepsilon ^2 n\log n$ edges.
\end{lemma}

\begin{proof}
Consider the following construction of a random subgraph $\Gamma _0$ of $\Gamma \setminus F$ with at most $5\varepsilon ^2n\log n$ edges:\\
For each $v\in SMALL(\Gamma )$ let $E_v$ be $E_{\Gamma}(v)$, and for each $v\in V(\Gamma )\setminus SMALL(\Gamma )$ let $E_v$ be a subset of $E_{\Gamma }(v)$ of size exactly $5\varepsilon ^2\log n$ , chosen uniformly at random, and set $E(\Gamma _0) = \bigcup_{v\in V(\Gamma)}E_v \ \setminus E(F)$.\\
\noindent Clearly, $\Gamma _0$ has at most $5\varepsilon ^2 n\log n$ edges, and has minimum degree at least $\min \{\delta (\Gamma ) - 2k+2 ,5\varepsilon ^2 \log n -2k+2 \}$, which is at least 2. We will show that it is also an $\left( \frac{n}{8},2 \right)$--expander with positive probability:\\
Let $U\subseteq V(\Gamma )$ be some vertex subset, and denote: $U_1 := U\cap SMALL(\Gamma )$, $U_2 := U\setminus U_1$, and $n_1,n_2$ the sizes of $U_1,U_2$, respectively. We will consider different cases for $n_2$.\\
First, we will show that for $n_2 \leq \frac{n}{\sqrt{\log n}}$ we have $|N_{\Gamma _0}(U)| \geq 2|U|$ with probability 1. By $(P3)$, we deduce the following properties:
\begin{itemize}
	\item $|N_{\Gamma _0}(U_1)| \geq 2n_1$;
	\item $| \left( U_1 \cup N_{\Gamma _0}(U_1) \right) \cap \left( U_2 \cup N_{\Gamma _0}(U_2) \right) | \le n_2$.
\end{itemize}
Consider two cases:
\begin{enumerate}
	\item $n_2<\varepsilon ^2 \log n -k$: Let $v\in U_2$, then $d_{\Gamma _0}(v) \geq 5\varepsilon ^2\log n -2k +2 \geq 5n_2$. So
		\begin{eqnarray*}
		|N_{\Gamma _0}(U)| & = & |N_{\Gamma _0}(U_2) \setminus U_1 | + |N_{\Gamma _0}(U_1) \setminus (U_2 \cup N_{\Gamma _0}(U_2))|\\
		& \geq & d_{\Gamma _0}(v) - 2n_2 + 2n_1 - n_2 \geq 2n_1 + 2n_2  = 2|U|.
	\end{eqnarray*}
	\item $\varepsilon ^2 \log n -k\leq n_2 \leq \frac{n}{\sqrt{\log n}}$: First we observe that $\left| N_{\Gamma _0}(U_2) \right| \geq 4n_2$. Assume otherwise, and let $W$ be some set of size $4n_2$ containing $N_{\Gamma _0}(U_2) $. By $(P4)$ we have
	\begin{eqnarray*}
		(5\varepsilon ^2 \log n -2k+2) \cdot n_2 & \leq & \sum _{u\in U_2}d_{\Gamma _0}(u) \leq 2e_{\Gamma _0}(U_2 \cup W) \leq 2e_{\Gamma}(U_2 \cup W) < \varepsilon ^2 \log n \cdot n_2
	\end{eqnarray*}
	--- a contradiction. So
	\begin{eqnarray*}
		|N_{\Gamma _0}(U)| & = & |N_{\Gamma _0}(U_2) \setminus U_1 | + |N_{\Gamma _0}(U_1) \setminus (U_2 \cup N_{\Gamma _0}(U_2))|\\
		& \geq & 4n_2 - n_2 + 2n_1 - n_2 = 2n_1 + 2n_2  = 2|U|.
	\end{eqnarray*}	
\end{enumerate}

\noindent Finally, we will bound from above the probability that there is a set $U$ with $\frac{n}{\sqrt{\log n}} \leq n_2 \leq \frac{n}{8}$ such that $|N_{\Gamma _0}(U)| \leq 2|U|$. If this is the case, then there is some set $W'$ of size $2|U|$ containing $N_{\Gamma _0}(U)$, and therefore, in particular, $e_{\Gamma _0}(U_2,V(\Gamma)\setminus (U_2\cup W)) = 0$, where $W:= W' \cup U_1$. Observe that
$$
|W| = |W'|+|U_1| = 2(|U_1|+|U_2|)+|U_1| =  2n_2 + 3n_1 \le 2n_2 + 3|SMALL(\Gamma )| \le 2.5n_2,
$$
and therefore, by $(P5)$, $e_{\Gamma}(U_2,V\setminus (U_2\cup W)) \geq 0.04n_2\log n$. We use this to bound the probability that such $U,W'$ exist.\\
\noindent Let $u\in U_2$. The probability that $E_u\cap E_{\Gamma }(u,V(\Gamma) \setminus (U_2\cup W)) \subseteq E_F(u)$ is at most
	$$
		\frac{ \binom{d_{\Gamma }(u)-d_{\Gamma }(u,V(\Gamma )\setminus (U_2\cup W)) + d_F(u)}{5\varepsilon ^2\log n} }{ \binom{d_{\Gamma }(u)}{5\varepsilon ^2\log n} } \leq \exp \left( - \frac{5\varepsilon ^2\log n}{\Delta (\Gamma )} \cdot \left( d_{\Gamma} \left( u,V(\Gamma )\setminus (U_2\cup W) \right) +2k-2 \right) \right) .
	$$
	As the events $E_u\cap E_{\Gamma }(u,V(\Gamma) \setminus (U_2\cup W) \subseteq E_F(u)$ are independent for distinct vertices $u$, multiplying the probabilities, and using the fact that since $\Gamma \in (P1)$ we have that $\Delta (\Gamma ) \leq 10 \log n$, we get
	\begin{eqnarray*}
	Pr\left[ e_{\Gamma _0}(U_2,V(\Gamma)\setminus (U_2\cup W)) = 0 \right] & = & \prod_{u\in U_2}Pr \left[E_u\cap E_{\Gamma }(u,V(\Gamma) \setminus (U_2\cup W) \subseteq E_F(u) \right] \\
	& \leq & \prod_{u\in U_2} \exp \left( - \frac{1}{2}\varepsilon ^2 \cdot \left( d_{\Gamma }(u,V(\Gamma )\setminus (U_2\cup W)) +2k-2 \right)\right) \\
	& = & \exp \left( - \frac{1}{2}\varepsilon ^2 \cdot \left( e_{\Gamma }(U_2,V(\Gamma )\setminus (U_2\cup W)) + (2k-2)|U_2| \right) \right) \\
	& = & \exp \left( - \omega \left( n \right) \right) ,
	\end{eqnarray*}
	
	\noindent where the last inequality is derived from the fact that $\Gamma \in (P5)$.\\
	\noindent Using the union bound, the probability that such $U,W$ exist is at most $3^n \cdot \exp \left( - \omega \left( n \right) \right) = o(1)$.\\
	\noindent In particular, with positive probability the random subgraph $\Gamma _0$ is indeed an expander as required, and therefore there exists such a subgraph of $\Gamma \setminus F$.
\end{proof}

\begin{lemma} \label{lemma1.2}
	Let $\Gamma$ be a graph on $n$ vertices such that $\delta \left( \Gamma \right) \geq 2k$ and $\Gamma \in (P1)$--$(P6)$, and let $F\subseteq \Gamma$ with $\Delta (F) \leq 2k-2$. Then $\Gamma \setminus F$ contains a subgraph $\Gamma _1$ which is a connected $\left( \frac{n}{8},2 \right)$--expander with at most $5\varepsilon ^2 n\log n + 1$ edges.
\end{lemma}

\begin{proof}
	The subgraph $\Gamma _0$ of $\Gamma \setminus F$ is an $\left( \frac{n}{8},2 \right)$--expander, and therefore has at most $2$ connected components, each of size at least $3n/8$.\\
	By $(P6)$, if $\Gamma _0$ is not connected, then $\Gamma \setminus F$ contains an edge between the two connected components. In particular, $\Gamma _1$ can be obtained by adding this edge, if necessary.
\end{proof}

\subsection{Proof of Lemma \ref{lemma2}}

\begin{lemma} \label{lammaBPs}
Let $\Gamma \subseteq G$ be a connected, non--Hamiltonian $\left( \frac{n}{8},2 \right)$--expander with at most $6\varepsilon ^2 n\log n$ edges, and let $F\subseteq G$ with $\Delta (F) \leq 2k-2$. then $G\setminus F$ contains a set $B\subseteq \binom{E(G)}{2}$ of booster pairs with respect to $\Gamma$, such that
\begin{itemize}
	\item $|B| \geq \frac{\varepsilon}{800}\cdot n^3$;
	\item Every edge of $E(G\setminus F)$ is a member of at most $\frac{n}{2}$ booster pairs in $B$.
\end{itemize}
\end{lemma}

\begin{proof}
We will use the fact that $\Gamma$ is expanding, and that $\delta (G) \geq \left( \frac{1}{2} + \varepsilon \right) n$, to build required \emph{BP} set as follows.\\
Let $P$ be a longest path in $\Gamma$, and let $S\subseteq V(G)$ be a set of $\frac{n}{8}$ starting vertices of paths obtained from $P$ by a sequence of rotations (such a set exists, by \cite{POS}).\\
For $s\in S$, denote by $P_s$ the path obtained from $P$ by rotations which caused $s$ to be added to $S$, and by $T_s$ a set of $\frac{n}{8}$ possible end vertices to paths obtained from $P_s$ by a sequence of rotations with $s$ fixed as the starting vertex.\\
Let $s\in S,t\in T_s$. For each such pair, we will add a set of \emph{BP}'s to $B$ (possibly some are already members of $B$). Consider two cases:
\begin{enumerate}
	\item The number of neighbours of $s$, and of $t$, in $G$ that are in $P$ is at least $\left( \frac{1}{2} + \frac{1}{3}\varepsilon \right) n$:\\
	By the pigeon--hole principle
	$$
	\left| \{ u\in P_s:\ (s,u),(p(u),t)\in E(G\setminus F) \} \right| \geq \frac{2}{3}\varepsilon n .
	$$
	Add the set of such edge pairs to $B$. Here, $p(u)$ denotes the predecessor of $u$ along $P_s$, as defined in Section \ref{sec-per}.
	
	\item The number of neighbours of either $s$ or $t$ in the graph $G\setminus F$ that are in $P$ is less than $\left( \frac{1}{2}+\frac{1}{3}\varepsilon \right) n$:\\
	In this case $s$ or $t$ has at least $\frac{2}{3} \varepsilon n$ edges going outside $P$. Let $e_1,...,e_{\frac{2}{3}\varepsilon n}$ be some subset of them, and let $f_1,...,f_{\frac{2}{3}\epsilon n}$ be some subset of edges touching the other end of the path. For every $i=1,...,\frac{2}{3} \varepsilon n$, add the pair $\{e_i,f_i\}$ to $B$
	
\end{enumerate}

\noindent In each of the two cases, for each $s\in S,t\in T_s$ we added at least $\frac{2}{3}\varepsilon n$ pairs of edges to $B$. Every pair was examined at most eight times, since each pair is considered only when one vertex in one edge of the pair corresponds to $s$, and one vertex of the other edge corresponds to $t$. So in total at least $\frac{1}{8}\cdot \sum _{s\in S}|T_s|\cdot \frac{2}{3}\varepsilon n = \frac{\varepsilon}{768}n^3$ distinct pairs were added. In addition, any edge incident to $s\in S$ was included in at most one pair for each $t\in T_s$ for every time it was considered, and vice versa, so overall every edge is a member of at most $\frac{n}{2}$ pairs in $B$.\\
Finally, remove from $B$ all the edge pairs that contain an edge of $E(F)$ or of $E(\Gamma )$. Since $|E(F)|\leq (2k-2)n, |E(\Gamma )| \leq 6\varepsilon ^2 n\log n$, and since every edge is a member of at most $\frac{n}{2}$ \emph{BP}'s in $B$, at most $O(n^2 \log n) = o(n^3)$ edge pairs were removed. The remaining pairs in $B$ are \emph{BP}'s with respect to $\Gamma$ which are edges of $G\setminus F$, and after the removal we remain with $|B| \geq \frac{\varepsilon}{800}n^3$, as required.
\end{proof}

\noindent We now aim to show that $G_p$ is likely to contain a booster pair. Let $X=\left( V,E\right)$ be a graph, such that
\begin{itemize}
	\item $|V|= |E(G)| \leq \binom{n}{2}$;
	\item $\Delta (X) \leq \frac{n}{2}$;
	\item $|E(X)| \geq \frac{\varepsilon}{800} \cdot n^3$. 
\end{itemize}
Let $S\subseteq V$. We say that a vertex $x\in V\setminus \left( S \cup N_X(S) \right)$ is \emph{$S$--useful} if $\left| N_X(x) \setminus N_X(S) \right| \geq \frac{\varepsilon}{4000}\cdot n$.

\begin{lemma} \label{lemmaUse}
Let $S\subseteq V$ be a subset such that $|S| \leq \frac{\varepsilon}{4000}\cdot n^2\ , \left| N_X(S) \right| \leq \frac{\varepsilon}{2000}\cdot n^2$, then the set $A\subseteq V(X)\setminus S$ of all $S$--useful vertices has cardinality at least $ \frac{\varepsilon}{4000}\cdot n^2$.
\end{lemma}

\begin{proof}
Since $|S| \leq \frac{\varepsilon}{4000}\cdot n^2\ , \left| N_X(S) \right| \leq \frac{\varepsilon}{2000}\cdot n^2$ and $\Delta (X)\leq \frac{n}{2}$, there are at most $\frac{3\varepsilon}{8000}\cdot n^3$ edges of $E(X)$ with at least one end in $S\cup N_X(S)$. So
$$e_X\left( V\setminus \left(S\cup N_X(S) \right) \right) \geq |E(X)|- \frac{3\varepsilon}{8000}\cdot n^3 \geq \frac{7\varepsilon}{8000} \cdot n^3.$$
On the other hand, 
$$e_X\left( V\setminus \left(S\cup N_X(S) \right) \right) \leq \Delta (X)\cdot |A| + \frac{\varepsilon}{4000}\cdot n\cdot \left| V\setminus (S \cup N_X(S) \cup A) \right| \leq \frac{n}{2}\cdot |A| + \frac{\varepsilon}{4000}\cdot \binom{n}{2} n.$$
Putting the two inequalities together we get the desired bound 
$$ |A| \geq \frac{2}{n} \left( \frac{7\varepsilon}{8000} \cdot n^3 - \frac{\varepsilon}{4000}\cdot \binom{n}{2} n \right) \geq \frac{\varepsilon}{4000}\cdot n^2.$$
\end{proof}

\noindent Let $X^p$ be the probability space of induced subgraphs of $X$ obtained by retaining each vertex of $X$ with probability $p$, or losing it (and all its edges) with probability $1-p$, independently of all other vertices.

\begin{lemma} \label{lemmaSpans}
Let $\frac{\log n}{n}\leq p \leq \frac{2\log n}{n}$. Then $Pr\left[ X^{p} \mbox{ spans no edge} \right] \leq 2\exp \left( -\frac{\varepsilon}{8100}\cdot n^2p \right)$.
\end{lemma}

\begin{proof}
Initialize $ S = \emptyset$. We will sample vertices of $V$ one after another in an adaptive order, and add successful vertices of $X^{p}$ to $S$. At each step, we will sample a vertex which we have not sampled previously and which is $S$--useful. We will do this until $\left| N_X(S) \right| \geq \frac{\varepsilon}{2000}\cdot n^2$, until no unsampled $S$--useful vertices remain, or until we have sampled $\frac{\varepsilon}{4000}\cdot n^2$ vertices (whichever comes first).
\begin{clm}
The probability that we stopped the process with $\left| N_X(S) \right| < \frac{\varepsilon}{2000}\cdot n^2$ is at most $\exp \left( -\frac{\varepsilon}{8100}\cdot n^2p \right)$.
\end{clm}

\begin{proof}
By Lemma \ref{lemmaUse}, if $V$ contains no unsampled $S$--useful vertices, where $\left| N_X(S) \right| < \frac{\varepsilon}{2000}\cdot n^2$, then at least $ \frac{\varepsilon}{4000}\cdot n^2$ vertices were sampled, which means that we stopped the process after sampling exactly $\frac{\varepsilon}{4000}\cdot n^2$ vertices. Since each successful $S$--useful vertex adds at least $ \frac{\varepsilon}{4000}\cdot n $ vertices to $N_X(S)$, among these samples there were at most $2n$ successes.\\
The probability that this occurs is at most $Pr \left[ Bin\left( \frac{\varepsilon}{4000}\cdot n^2,p \right) < 2n \right]$, and using Chernoff's bound for binomial tails the claim follows.
\end{proof}

\noindent To complete the proof, observe that the probability that $X^{p}$ spans no edge is at most the probability that either we failed at sampling $V$ and finding the desired set $S$, or $N_X(S)\cap X^{p} = \emptyset$.\\
In the process of creating $S$ we sampled at most $\frac{\varepsilon}{4000}\cdot n^2$ vertices, and upon ending it successfully we are left with $\left| N_X(S) \right| \geq \frac{\varepsilon}{2000}\cdot n^2$, which means that there are at least $\frac{\varepsilon}{4000}\cdot n^2$ unsampled vertices in $N_X(S)$. So
$$
Pr\left[ N_X(S)\cap X^{p} = \emptyset\ |\mbox{ sampling succeeded} \right] \leq Pr \left[ Bin\left( \frac{\varepsilon}{4000}\cdot n^2,p \right) =0 \right] \leq \exp \left( -\frac{\varepsilon}{8100}\cdot n^2p \right).
$$
Putting the two bounds together we get
\begin{eqnarray*}
Pr\left[ X^{p} \mbox{ contains no edge} \right] & \leq & Pr\left[ \mbox{sampling failed} \right]+Pr\left[ N_X(S)\cap X^{p} = \emptyset\ |\mbox{ sampling succeeded} \right]\\
& \leq & 2\exp \left( -\frac{\varepsilon}{8100}\cdot n^2p \right).
\end{eqnarray*}
\end{proof}

\begin{corol}
Let $\frac{\log n}{n}\leq p \leq \frac{2\log n}{n}$, let $\Gamma \subseteq G$ be a connected, non--Hamiltonian $\left( \frac{n}{8},2 \right)$--expander with at most $6\varepsilon ^2 n\log n$ edges, and let $F\subseteq G$ with $\Delta (F) \leq 2k-2$. Then the probability that $G_{p}\setminus F$ contains no booster pair with respect to $\Gamma$ is at most $2\exp \left( -\frac{\varepsilon}{8100}\cdot n^2p \right)$.
\end{corol}

\begin{proof}
Let $B$ be a set of booster pairs with respect to $\Gamma$ in $G\setminus F$ such that:
\begin{itemize}
	\item $|B| \geq \frac{\varepsilon}{800}\cdot n^3$;
	\item Every edge of $E(G)$ is a member of at most $\frac{n}{2}$ booster pairs in $B$.
\end{itemize}
By Lemma \ref{lammaBPs}, such $B$ exists.
Construct an auxiliary graph $X_{\Gamma}$ as follows:\\
First define $V(X_{\Gamma}) := E(G)$. For each $e,f\in E(G \setminus F)$: add $(e,f)$ to $E(X_{\Gamma})$ if $\{e,f\} \in B$.\\
By the definition of $B$, $X_{\Gamma}$ is a graph fulfilling the requirements of Lemma \ref{lemmaUse}. So $$Pr\left[ X_{\Gamma}^{p} \mbox{ contains no edge} \right] \leq 2\exp \left( -\frac{\varepsilon}{8100}\cdot n^2p \right).$$ Since an edge of $X_{\Gamma}^p$ corresponds to a booster pair in $G_p \setminus F$, the corollary follows. \\
\end{proof}

\begin{corol}
Let $\frac{\log n}{n}\leq p \leq \frac{2\log n}{n}$. Then the following holds with probability $1-n^{-\omega (1)}$:\\
For every subgraph $\Gamma $ of $G_{p}$ which is a connected, non--Hamiltonian $\left( \frac{n}{8},2\right)$--expander with at most $6\varepsilon ^2 n\log n$ edges, and for every subgraph $F\subseteq G_p$ such that $\Delta (F) \leq 2k-2$, $G_{p} \setminus F$ contains a booster pair with respect to $\Gamma$.
\end{corol}

\begin{proof}
Let $A_{F,\Gamma}$ be the event that $G_p \setminus F$ does not contain a booster pair with respect to $\Gamma$. Using the union bound, we bound the probability that there are $F,\Gamma \subseteq G_p$ for which $A_{F,\Gamma}$ holds.

	\begin{eqnarray*}
		Pr(\exists F,\Gamma \subseteq G_p \ s.t.\ A_{F,\Gamma}) & \leq &  \sum\limits _{i=1}^{6\varepsilon ^2 n\log n} \sum\limits_{j=0}^{(k-1)n} \binom{|E(G)|}{i} \binom{|E(G)|}{j} p^{i+j} \cdot 2\exp \left( -\frac{\varepsilon}{8100}\cdot n^2p \right)\\
		& \leq & \sum\limits _{i=1}^{6\varepsilon ^2 n\log n} \sum\limits_{j=0}^{(k-1)n} {\left( \frac{en^2p}{2i} \right)}^i {\left( \frac{en^2p}{2j} \right)} ^{j} \cdot 2\exp \left( -\frac{\varepsilon}{8100}\cdot n^2p \right)\\
		& \leq & e^{O(n \log \log n)}\cdot {\left( \frac{e}{6\varepsilon ^2} \right)}^{6\varepsilon ^2 n\log n} \cdot 2\exp \left( -\frac{\varepsilon}{8100} \cdot n\log n \right)\\
		& = & n^{-\omega (1)}\ ,
	\end{eqnarray*}
	
\noindent where the last equality holds when $\varepsilon$ is small enough such that $\varepsilon ^2 \cdot \log \left( \frac{1}{\varepsilon} \right) \ll \varepsilon$.
\end{proof}

With Corollary 2 we can now complete the proof. Let $A_{m}$ be the event that $G_m $ contains a sparse expander $\Gamma$ and a subgraph $F$ with maximum degree at most $2k-2$, such that $G_m \setminus F$ does not contain a booster pair with respect to it $\Gamma$. Then by Lemma \ref{lem-asym}
$$
Pr\left[ A_{\tau _{2k}} \right] \leq o(1)+\sum _{m=m_1}^{m_2} Pr\left[ A_{m} \right] \leq o(1)+(m_2-m_1)\cdot 3 \sqrt{|E(G)|} \cdot n^{-\omega (1)} = o(1).
$$

\section{$(\alpha ,\beta )$-dense graphs} \label{sec-abdense}

In this section we provide a proof of Theorem \ref{thmJumbled}. Throughout the section, $G$ is some fixed graph assumed to be $( \alpha , \beta )$--dense (see Def. \ref{def-abdense}).

We take a similar approach to the proof of Theorem \ref{thmDirac}, showing that with high probability for every $F\subseteq G_{\tau _{2k}}$ with $\Delta (F) \leq 2k-2$, the graph $G_{\tau _{2k}} \setminus F$ is Hamiltonian, by proving two main lemmas:

\begin{lemma} \label{lemma3}
With high probability for every $F\subseteq G_{\tau _{2k}}$ with $\Delta (F) \leq 2k-2$, $G_{\tau _{2k}} \setminus F$ contains a subgraph $\Gamma _0$ with at most $7\beta ^2 n\log n$ edges, which has the following properties:
\begin{itemize}
	\item $\Gamma _0$ is connected;
	\item $\Gamma _0$ is an $\left( \frac{\alpha}{3}\cdot n ,2\right)$--expander;
	\item For every two disjoint subsets $U,W\subseteq V(G)$ of size $|U|,|W|\geq \alpha n$, there is a matching in $\Gamma _0$ between $U$ and $W$ of size at least $\frac{\alpha ^3}{100}\beta n$.
\end{itemize}
\end{lemma}

\begin{lemma} \label{lemma4}
With high probability for every $F\subseteq G_{\tau _{2k}}$ with $\Delta (F) \leq 2k-2$ and for every subgraph $\Gamma $ of $G_{\tau _{2k}} \setminus F$ with at most $8\beta ^2 n\log n$ edges, such that 
\begin{itemize}
	\item $\Gamma$ is connected;
	\item $\Gamma$ is an $\left( \frac{\alpha}{3}\cdot n ,2\right)$--expander;
	\item For every two disjoint subsets $U,W\subseteq V(G)$ of size $|U|,|W|\geq \alpha n$, there is a matching in $\Gamma$ between $U$ and $W$ of size at least $\frac{\alpha ^3}{100}\beta n$;
\end{itemize}
the graph $G_{\tau _{2k}}\setminus F$ contains a booster pair with respect to $\Gamma$.
\end{lemma}

Theorem \ref{thmJumbled} is obtained from Lemma \ref{lemma3} and Lemma \ref{lemma4} in the same way Theorem \ref{thmDirac} was obtained in the previous section.

Throughout the proofs we will assume that $\beta$ is sufficiently small relative to $\alpha$ without stating so explicitly. We also assume that $\alpha < \frac{1}{4}$, since the complementing case when $\delta (G) \geq \left( \frac{1}{2} + \beta \right)n$ is already covered by Theorem \ref{thmDirac}.

\subsection{Proof of Lemma \ref{lemma3}}

\noindent For a graph $\Gamma$ on $n$ vertices, let $$d_0:=5\beta ^2 \log n,$$ and $$SMALL(\Gamma ):=\lbrace v\in V(\Gamma ): d_{\Gamma}(v)\leq d_0 \rbrace.$$ We will first show that WHP $G_{\tau _{2k}}$ has the following properties:
\begin{itemize}
	\item[(Q1)] $\Delta \left( G_{\tau _{2k}} \right) \leq 10\alpha ^{-1}\log n$;
	\item[(Q2)] $\left| SMALL \left(G_{\tau _{2k}} \right) \right| \leq n^{0.1}$;
	\item[(Q3)] $\forall u,v\in SMALL \left(G_{\tau _{2k}} \right) : dist_{G_{\tau _{2k}}}(u,v) > 4$;
	\item[(Q4)] $\forall U\subset V(G)\ s.t.\ 4\beta ^2\log n\leq |U| \leq \frac{5n}{\sqrt{\log n}}:e_{G_{\tau _{2k}}}(U)<2\beta ^2|U|\log n$;
	\item[(Q5)] $\forall U,W\subset V(G)\ disjoint\ s.t.\ \frac{n}{\sqrt{\log n}} \leq |U| \leq \frac{\alpha}{3}\cdot n,|W|=2|U|:e_{G_{\tau _{2k}}}(U,V(G)\setminus (U\cup W)) \geq 0.4\alpha |U|\log n$;
	\item[(Q6)] $\forall U\subset V(G)\ s.t.\ \alpha n \leq |U|\leq  \frac{1}{2} n: e_{G_{\tau _{2k}} }\left( U, V(G)\setminus U \right) > n\sqrt{\log n}$;
	\item[(Q7)] $\forall U,W\subset V(G)\ disjoint\ s.t.\ |U|,|W| \geq \alpha n : e_{G_{\tau _{2k}} }\left( U, W \right) > \frac{\alpha ^2}{3} \beta n \log n$.
\end{itemize}

\begin{proof} For each property, we will bound the probability of $G_{\tau _{2k}}$ failing to have it separately.

Let $N = |E(G)|$, and let
$$p_1 = \frac{\log n}{2n},\ p_2 = \frac{\log n}{\alpha n},\ m_1 = N\cdot p_1,\ m_2 = N\cdot p_2.$$
Let $p^*$ be the unique solution in the interval $[0,1]$ to the equation $$ \sum _{v\in V(G)} (1-x)^{d_G(v)} = 1,$$
and let $$p^*_1 = p^*-\frac{\sqrt{\log n}}{n},\ p^*_2 = p^*+\frac{\sqrt{\log n}}{n},\ m^*_1 = N\cdot p^*_1,\ m^*_2 = N\cdot p^*_2.$$
Observe that

\begin{eqnarray*}
	\sum _{v\in V(G)} (1-p^*_1)^{d_G(v)} & = & \sum _{v\in V(G)} (1-p^*)^{d_G(v)} \cdot \left( \frac{1-p^*_1}{1-p^*} \right) ^{d_G(v)} \\
	& \leq & \sum _{v\in V(G)} (1-p^*)^{d_G(v)} \cdot \left( 1 + \frac{p^*-p^*_1}{1-p^*} \right) ^{d_G(v)} \\
	& \leq & \sum _{v\in V(G)} (1-p^*)^{d_G(v)} \cdot \left( 1 + \frac{2\sqrt{\log n}}{n} \right) ^{n} \\
	& \leq & \exp \left( 2\sqrt {\log n} \right) ,
\end{eqnarray*}
and similarly $\sum _{v\in V(G)} (1-p^*_2)^{d_G(v)} \geq \exp \left( -\alpha \sqrt {\log n} \right) .$

\begin{lemma} \label{lemmaTAUjumb}
It holds that $p_1 \leq p^*_1 \leq p^* \leq p^*_2 \leq p_2$. Furthermore, with high probability $m^*_1 \leq \tau _{2k} \leq m^*_2$.
\end{lemma}

\begin{proof}
Since the expression $\sum _{v\in V(G)} (1-x)^{d_G(v)}$ is decreasing in the interval $[0,1]$, in order to show that $p_1 \leq p^*_1 \leq p^*_2 \leq p_2$ it suffices to show that $$ \sum _{v\in V(G)} \left( 1-p_1 - \frac{\sqrt{\log n}}{n} \right)^{d_G(v)} \geq 1, \quad \sum _{v\in V(G)} \left( 1-p_2 + \frac{\sqrt{\log n}}{n} \right)^{d_G(v)} \leq 1.$$
We bound both expressions:
	$$
	\sum _{v\in V(G)} \left( 1-p_1 - \frac{\sqrt{\log n}}{n} \right)^{d_G(v)} \geq n\cdot \left( 1-\frac{3\log n}{4n} \right) ^n = \omega (1),
	$$
	$$
	\sum _{v\in V(G)} \left( 1-p_2 + \frac{\sqrt{\log n}}{n} \right)^{d_G(v)} \leq n\cdot \left( 1-\frac{\log n}{2\alpha n} \right) ^{(2\alpha + \beta )n} = o(1),
	$$
as desired.
	
Next, in order to show that with high probability $m^*_1 \leq \tau _{2k} \leq m^*_2$, it suffices to show that
\begin{enumerate}
	\item $Pr\left[ \delta \left( G_{p^*_1} \right) \geq 2k \right] = o(1)$;
	\item $Pr\left[ \delta \left( G_{p^*_2} \right) < 2k \right] = o(1)$.
\end{enumerate}

We will prove both bounds:

\begin{enumerate}
	\item We will bound the probability using Chebyshev's inequality. Let $X$ be the random variable counting the number of vertices in $V(G)$ with degree less than $2k$ in $G_{p^*_1}$. Recall that $p^*$ is such that $\sum _{v\in V(G)} \left( 1-p^* \right) ^{d_G(v)} = 1$, and that $p_1 \leq p^*_1 \leq p^*$. First, bound $\mathbb{E}[X]$ from below:
	
	\begin{eqnarray*}
	\mathbb{E}[X] & = & \sum _{v\in V(G)} Pr\left[ d_{G_{p^*_1}}(v) < 2k \right] = \sum _{v\in V(G)} \sum\limits_{i=0}^{2k-1} \binom{d_G(v)}{i}{p^*_1}^i(1-p^*_1)^{d_G(v)-i} \\
	& \geq & \sum _{v\in V(G)}  \left( \frac{\alpha np^*_1}{2k}\right) ^{2k-1}\cdot (1-p^*_1)^{d_G(v)} \geq \left( \frac{\alpha np_1}{2k}\right) ^{2k-1} = \omega (1).
	\end{eqnarray*}
	
	Bound $Var[X]$ from above:
	\begin{eqnarray*}
	Var[X] & = & \mathbb{E}\left[ X^2 \right] - \mathbb{E}\left[ X \right] ^2 \\
	& = & \sum _{u\neq v \in V(G)} Pr\left[ d_{G_{p^*_1}}(v) , d_{G_{p^*_1}}(u) < 2k \right] - Pr\left[ d_{G_{p^*_1}}(v) < 2k \right] \cdot Pr\left[  d_{G_{p^*_1}}(u) < 2k \right] \\
	& &  + \sum _{v\in V(G)} Pr\left[ d_{G_{p^*_1}}(v) < 2k \right] - \left( Pr\left[ d_{G_{p^*_1}}(v) < 2k \right] \right) ^2 .\\
	\end{eqnarray*}
	
	We separate the terms $u\neq v\in V(G)$ for which $(u,v)\in E(G)$ from the terms with $(u,v)\notin E(G)$, and observe that in the latter case the events $d_{G_{p^*_1}}(u)<2k$ and $d_{G_{p^*_1}}(v)<2k$ are independent.
	\begin{eqnarray*}
	& \leq & \mathbb{E}[X] +  \sum _{(u, v) \in E(G)} Pr\left[ d_{G_{p^*_1}}(v) , d_{G_{p^*_1}}(u) < 2k \right] - Pr\left[ d_{G_{p^*_1}}(v) < 2k \right] \cdot Pr\left[  d_{G_{p^*_1}}(u) < 2k \right] \\
	& \leq & \mathbb{E}[X] + \sum _{(u, v) \in E(G)} p^*_1\cdot Pr\left[ d_{G_{p^*_1}}(v) , d_{G_{p^*_1}}(u) < 2k | (u,v)\in E\left( G_{p^*_1} \right) \right] \\
	& \leq & \mathbb{E}[X] + \sum _{(u, v) \in E(G)} p_2\cdot \sum _{i=0}^{2k-2} \sum _{j=0}^{2k-2}\binom{n}{i}\binom{n}{j}{p_2}^{i+j} (1-p^*_1) ^{d_G(v)+d_G(u) -i -j -2} \\
	& \leq & \mathbb{E}[X] + \sum _{(u, v) \in E(G)} p_2\cdot 4k^2 \left( \frac{e\log n}{\alpha \cdot (2k-2)(1-p_2)} \right) ^{4k-4} (1-p^*_1) ^{d_G(v)+d_G(u) -2} \\
	& = & \mathbb{E}[X] + O\left( p_2 \cdot \log ^{4k-4} n \cdot \exp \left( 4 \sqrt{\log n} \right) \right) \cdot \sum _{(u, v) \in E(G)} (1-p^*) ^{d_G(v)+d_G(u)} \\
	& \le &  \mathbb{E}[X] + o(1)\cdot \left( \sum _{v\in V(G)} (1-p^*) ^{d_G(v)} \right) ^2 \\
	& = & \mathbb{E}[X] + o(1).
	\end{eqnarray*}
	
	Now, $\frac{Var[X]}{\mathbb{E}[X]^2} = O\left( \mathbb{E}[X] ^{-1} \right) = o (1),$ and by Chebyshev's inequality we get $$Pr\left[ \delta \left( G_{p^*_1} \right) \geq 2k \right] = Pr[X=0] = o(1).$$
	
	\item Let $v\in V(G)$. The probability that $d_{G_{p^*_2}}(v) < 2k$ is at most
	\begin{eqnarray*}
	Pr\left[ d_{G_{p^*_2}}(v) <2k \right] & = & \sum\limits_{i=0}^{2k-1} \binom{d_G(v)}{i}{p^*_2}^i(1-p^*_2)^{d_G(v)-i} \\
	& \leq & 2k\cdot \left( \frac{enp^*_2}{(2k-1)(1-p^*_2)} \right) ^{2k-1} \cdot (1-p^*_2)^{d_G(v)} \\
	& \leq & O\left( \log ^{2k} n \right) \cdot \exp \left( -\alpha \sqrt{\log n} \right) \cdot (1-p^*)^{d_G(v)} \\
	& = & o\left( (1-p^*)^{d_G(v)} \right) .
	\end{eqnarray*}
	By the union bound we get
	$$Pr\left[ \delta \left( G_{p^*_2} \right) < 2k \right] \leq \sum _{v\in V(G)} Pr\left[ d_{G_{p^*_2}}(v) <2k \right] = o\left( \sum _{v\in V(G)} (1-p^*)^{d_G(v)} \right) = o(1).$$
\end{enumerate}

\end{proof}

\noindent Now we show that with high probability $G_{\tau _{2k}}$ has properties $(Q1)-(Q6)$. With Lemma \ref{lemmaTAUjumb}, for a monotone increasing property $(Q)$ it suffices to show that $Pr\left( G_{p_1} \notin (Q) \right) = o(1)$ or that $Pr\left( G_{p^*_1} \notin (Q) \right) = o(1)$, and similarly for decreasing properties and $G_{p_2},\ G_{p^*_2}$.

We omit the calculations for properties $(Q1),(Q4)$, since they are almost identical to the calculations for their counterpart properties $(P1),(P4)$ in the proof of Lemma \ref{lemma1}.

\begin{itemize}
	
	\item[(Q2)] Use Markov's inequality:
	
	\begin{eqnarray*}
		Pr\left[ G_{p^*_1} \notin \bf{(Q2)} \right]  & \leq & \frac{\mathbb{E}\left[ |SMALL \left(G_{p^*_1} \right) | \right] }{n^{0.1}} \\
		& = & n^{-0.1} \sum _{v\in V(G)} Pr\left[ Bin(d_G(v),p^*_1) \leq 5\beta ^2\log n \right] \\
		& \leq & n^{-0.1}\cdot 5\beta ^2\log n \cdot \left( \frac{enp^*_1}{5\beta ^2\log n} \right) ^{5\beta ^2\log n} \cdot \sum _{v\in V(G)} (1-p^*_1)^{d_G(v)} \\
		& \leq & \exp \left( \left( -0.1 + 5\beta ^2 \log \left( \frac{e}{5\alpha \beta ^2} \right) \right) \log n +O\left( \log \log n + \sqrt{\log n} \right) \right) \\
		& = & o(1).
	\end{eqnarray*}
	
	\item[(Q3)] As $m^*_1 \leq \tau _{2k} \leq m^*_2$ with high probability, it is sufficient to prove that with high probability $G_{m^*_2}$ does not contain a path of length at most 4 between two (not necessarily distinct) members of $SMALL\left( G_{m^*_1} \right)$.\\ 
	We will prove this by showing that:
	\begin{enumerate}
		\item With high probability $G_{m^*_1}$ does not contain a path of length at most 4 between two (not necessarily distinct) members of $SMALL\left( G_{m^*_1} \right)$;
		\item With high probability, adding $m^*_2-m^*_1 \leq n \sqrt{\log n}$ random edges of $G\setminus G_{m^*_1}$ to $G_{m^*_1}$ does not result in a path between two members of $SMALL\left( G_{m^*_1} \right)$.
	\end{enumerate}
	
	The probability that a specific path $P$ of length $\ell$ is in $G_{m^*_1}$ is at most
	$$
	Pr \left[ P \in G_{m^*_1} \right] \leq \frac{ \binom{N-\ell}{Np_2-\ell} }{ \binom{N}{Np_2} } \leq \left( \frac{Np_2}{N} \right) ^{\ell} = \left( \frac{\log n}{\alpha n} \right) ^{\ell}.
	$$
	Now bound the probability that the two endpoints of a path $P$, denoted as $s,t$, are members of $SMALL\left( G_{m^*_1} \right)$, conditioned on the event $P \in G_{m^*_1}$. This probability is at most the probability that the vertex set $\{ s,t\}$ has at most $d_1 := 2\cdot d_0 -2$ edges between it and $V(G)\setminus \{s,t\}$ in $G_{m^*_1}$, not including the two edges belonging to $P$.
	
	\begin{eqnarray*}
		Pr\left[ s,t\in SMALL\left( G_{m^*_1}  \right) | P\in G_{m^*_1} \right] & \leq & \sum\limits_{i=0}^{d_1} \binom{2n-4}{i}\cdot \frac{ \binom{N-\ell -(d_G(s)+d_G(t) -2 )}{m_1 -\ell -i} }{\binom{N-\ell}{m_1-\ell}} \\
		& \leq & (d_1+1)\binom{2n}{d_1}\cdot \frac{ \binom{N-\ell -(d_G(s)+d_G(t) -2 )}{m_1 -\ell -d_1} }{\binom{N-\ell}{m_1-\ell}} \\
		& \leq & d_1\left( \frac{2en}{d_1} \right)^{d_1}\cdot \left( \frac{m_1-\ell}{N-\ell} \right)^{d_1}\\
		& & \cdot \left( 1-\frac{m^*_1-\ell -d_1}{N-\ell -d_1} \right) ^{d_G(s) + d_G(t) -2- d_1} \\
		& \leq & \left( \frac{2}{\beta ^2\alpha} \right)^{-10\beta ^2 \log n} \cdot (1-p^*)^{d_G(s)+d_G(t)} .
	\end{eqnarray*}
	
	Apply the union bound to bound the probability that there is a path $P\in G_{m_1}$ of length $1\leq \ell \leq 4$, such that both of its endpoints are in $SMALL\left( G_{m_1} \right)$:
	
	\begin{eqnarray*}
		Pr\left[ \exists P\in G_{m_1} :\ s,t\in SMALL\left( G_{m_1} \right) \right]  & \leq &  \sum_{s,t\in V(G)} \sum_{\ell =1}^4 n^{\ell -1} \cdot \left( \frac{\log n}{\alpha n} \right) ^{\ell} \\
		& & \cdot Pr\left[ s,t\in SMALL\left( G_{m^*_1}  \right) | P\in G_{m^*_1} \right] \\
		& \le & \sum_{s,t\in V(G)} (1-p^*)^{d_G(s)+d_G(t)} \\
		& & \cdot \exp \left( \log n \cdot \left( -1 -10\beta ^2 \log \left( \frac{2}{\beta ^2\alpha} \right) + o(1) \right) \right) \\
		& \leq & n^{-0.9} = o(1),
	\end{eqnarray*}
	
	for small enough $\beta$.
	
	Finally, we bound the probability that adding $n \sqrt{\log n}$ random edges to $G_{m^*_1}$ results in a short path between two vertices of $SMALL\left( G_{m^*_1} \right)$. Since we already proved that with high probability $G_{m^*_1} \in (Q2)$ and $G_{m^*_2} \in (Q1)$, we can assume for the sake of calculation that indeed $G_{m^*_1},G_{m^*_2}$ have these properties.\\
	For the set of $n \sqrt{\log n}$ added edges to close a path, at least one of the edges must have both its vertices inside the set of vertices at distance at most 3 from $SMALL\left( G_{m^*_1} \right)$ in $G_{m^*_2}$ . By $(Q1),(Q2)$, this set is of size at most $n^{0.1}\cdot (10\alpha ^{-1}\log n)^3 \leq n^{0.2}$. By the union bound, the probability that at least one of the added edges is in this set is at most
	$$
	n \sqrt{\log n}\cdot \frac{ \binom{n^{0.2}}{2} }{N-m^*_2} = o(1).
	$$

	\item[(Q5)] For sets $U,W\subseteq V(G)$ such that $\frac{n}{\sqrt{\log n}} \leq |U| \leq \frac{\alpha}{3}n$ and $|W|=2|U|$ we have $|U \cup W| \leq \alpha n$, so $$e_G(U,V(G)\setminus (U\cup W)) \geq |U|\cdot (\delta (G) - \alpha n) \geq \alpha |U| n.$$ So
		\begin{eqnarray*}
		Pr\left[ G_{p_1} \notin {\bf (Q5)} \right] & \leq & \sum\limits_{i=\frac{n}{\sqrt{\log n}}}^{\frac{\alpha}{3} n} \binom{n}{i} \binom{n}{2i} Pr\left[ Bin\left( \alpha i n , p_1 \right) \leq 0.4 \alpha i \log n \right] \\
		& \leq & 3^n\cdot \exp \left( -\Omega \left( n^2p_1/\sqrt{\log n} \right) \right) = o(1) .
	\end{eqnarray*}
	
	\item[(Q6)] It is easy to see that if a graph has property $(Q7)$ then it also has property $(Q6)$, and so it suffices to prove that the former happens with high probability for  $G_{\tau_{2k}}$.
	
	\item[(Q7)] Since $G$ is $(\alpha ,\beta )$--dense, if $|U|,|W| \geq \alpha n$ we have $e_G(U,W) \geq \alpha ^2 \beta n^2$. So
	\begin{eqnarray*}
		Pr\left[ G_{p_1} \notin {\bf (Q7)} \right] & \leq & 3^n \cdot Pr\left[ Bin\left( \alpha ^2 \beta n^2 , p_1 \right) \leq \frac{\alpha ^2}{3}\beta n\log n \right] \\
		& \leq & 3^n\cdot \exp \left( - \Omega \left( n \log n \right) \right) = o(1) .
	\end{eqnarray*}
\end{itemize}
\end{proof}

\begin{lemma} \label{lemma3.1}
	Let $\Gamma$ be a graph on $n$ vertices such that $\delta \left( \Gamma \right) \geq 2k$ and such that $\Gamma \in (Q1)$--$(Q5)$, and let $F\subseteq \Gamma$ with $\Delta (F) \leq 2k-2$. Then $\Gamma \setminus F$ contains a subgraph $\Gamma ^{(1)}$ which is an $\left( \frac{\alpha}{3}n,2 \right)$--expander with at most $5\beta ^2 n\log n$ edges.
\end{lemma}

\begin{proof}
The proof is essentially identical to the proof of Lemma \ref{lemma1.1}, up to some constants being different.
\end{proof}

\begin{lemma} \label{lemma3.2}
	Let $\Gamma$ be a graph on $n$ vertices such that $\delta \left( \Gamma \right) \geq 2k$ and such that $\Gamma \in (Q1)$--$(Q6)$, and let $F\subseteq \Gamma$ with $\Delta (F) \leq 2k-2$. Then $\Gamma \setminus F$ contains a subgraph $\Gamma ^{(2)}$ which is a connected $\left( \frac{\alpha}{3}n,2 \right)$--expander with at most $6\beta ^2 n\log n $ edges.
\end{lemma}

\begin{proof}
	This proof follows a similar line to the proof of Lemma \ref{lemma1.2}.\\
	The subgraph $\Gamma ^{(1)}$ of $\Gamma \setminus F$ is an $\left( \frac{\alpha}{3}n,2 \right)$--expander, and therefore has at most $\alpha ^{-1}$ connected components, each of size at least $\alpha n$.\\
	By $(Q6)$, if $\Gamma ^{(1)}$ is not connected, then $\Gamma \setminus F$ contains an edge between any connected component to the rest of the graph. In particular, $\Gamma ^{(2)}$ can be obtained by sequentially adding such edges to $\Gamma ^{(1)}$, if necessary, at most $\alpha ^{-1}$ times.
\end{proof}

\begin{lemma} \label{lemma3.3}
	Let $\Gamma$ be a graph on $n$ vertices such that $\delta \left( \Gamma \right) \geq 2k$, $|E(\Gamma )| = m_1$ and such that $\Gamma \in (Q1),(Q7)$, and let $F\subseteq \Gamma$ with $\Delta (F) \leq 2k-2$. Then $\Gamma \setminus F$ contains a subgraph $\Gamma ^{(3)}$ which has at most $\beta ^2 n\log n$ edges, such that for every two disjoint subsets $U,W\subseteq V(G)$ of size $|U|,|W|\geq \alpha n$, there is a matching in $\Gamma ^{(3)}$ between $U$ and $W$ of size at least $\frac{\alpha ^3}{100}\beta n$.
\end{lemma}

\begin{proof}
Let $\Gamma ^{\prime} := \Gamma \setminus F$, and let $\Gamma ^{(3)}$ be a random subgraph of $\Gamma ^{\prime}$ distributed according to $\Gamma ^{\prime}_{\beta ^2}$. We will show that with positive probability $\Gamma ^{(3)}$ satisfies the desired properties, and therefore show that the desired subgraph indeed exists.\\
First, we bound the probability that for some disjoint subsets $U,W\subseteq V(G)$ the maximum matching size $\nu _{\Gamma ^{(3)}} (U, W)$ between them is at most $\frac{\alpha ^3}{100}\beta n$. Observe that for some $i \leq \frac{\alpha ^3}{100}\beta n$, if $\nu _{\Gamma ^{(3)}} (U, W) = i$ then there are some $i$ edges in $E_{\Gamma ^{(3)}}(U,W)$ such that the set of all their (exactly) $2i$ end vertices is a cover of all the edges in $E_{\Gamma ^{(3)}}(U,W)$ -- meaning that all edges of $\Gamma ^{\prime}$ that are not touching any of these $2i$ vertices are non--edges in $\Gamma ^{(3)}$. Since $\Gamma \in (Q1),(Q7)$ we get
\begin{eqnarray*}
	Pr \left[ \nu _{\Gamma ^{(3)}} (U, W) \leq \frac{\alpha ^3}{100}\beta n \right] & = & \sum\limits_{i=0}^{\frac{\alpha ^3}{100}\beta n} Pr \left[ \nu _{\Gamma ^{(3)}} (U, W) =i \right] \\
	& \leq & \sum\limits_{i=0}^{\frac{\alpha ^3}{100}\beta n} \binom{e_{\Gamma}(U,W)}{i}\beta ^{2i} \left( 1-\beta ^2 \right) ^{e_{\Gamma ^{\prime}}(U,W) - 2i\cdot \Delta (\Gamma )} \\
	& \leq & \sum\limits_{i=0}^{\frac{\alpha ^3}{100}\beta n} \binom{e_{\Gamma}(U,W)}{i} \left( 1-\beta ^2 \right) ^{e_{\Gamma}(U,W) -2(k-1)n - 2i\cdot \Delta (\Gamma )} \\
	& \leq & n \cdot \binom{\frac{\alpha ^2}{3}\beta n\log n}{\frac{\alpha ^3}{100}\beta n} \left( 1-\beta ^2 \right) ^{\frac{\alpha ^2}{3}\beta n\log n  -2kn - 2\cdot \frac{\alpha ^3}{100}\beta n \cdot \alpha ^{-1}10\log n} \\
	& \leq & \exp \left( O\left( n \log \log n \right) - \Omega \left(n \log n \right) \right)\\
	& = & \exp \left( - \Omega \left(n \log n \right) \right) .
\end{eqnarray*}
Summing over all (at most $3^n$) possible subsets $U$ and $W$ we derive that the estimated probability is of order $o(1)$.\\
To finish the proof, observe that the probability that $\Gamma ^{(3)}$ has more than $\beta ^2 n\log n$ edges, which is more than $\beta ^2 \cdot (m_1 - (k-1)n)$, is at most $$Pr \left[ Bin\left( m_1 - (k-1)n,\beta ^2 \right) > \beta ^2 n\log n \right]  \leq \frac{1}{2}.$$
So overall, the probability that $\Gamma ^{(3)}$ satisfies both conditions is at least $\frac{1}{2} - o(1) > 0$.
\end{proof}

\noindent To finish the proof of Lemma \ref{lemma3}, observe that by Lemma \ref{lemma3.2}, with high probability for every $F\subseteq G_{\tau _{2k}}$ with $\Delta (F) \leq 2k-2$, $G_{\tau _{2k}} \setminus F$ contains a subgraph $\Gamma ^{(2)}$ which is a connected $\left( \frac{\alpha}{3}n,2 \right)$--expander with at most $6\beta ^2 n\log n $ edges. Furthermore, by Lemma \ref{lemma3.3} the subgraph $G_{m_1}\setminus F$, and therefore $G_{\tau _{2k}}\setminus F$, contains a subgraph $\Gamma ^{(3)}$ which has at most $\beta ^2 n\log n$ edges, such that for every two disjoint subsets $U,W\subseteq V(G)$ of size $|U|,|W|\geq \alpha n$, there is a matching in $\Gamma ^{(3)}$ between $U$ and $W$ of size at least $\frac{\alpha ^3}{100}\beta n$.\\
Now simply set $\Gamma _3 := \Gamma ^{(2)} \cup \Gamma ^{(3)}$ to obtain the lemma.

\subsection{Proof of Lemma \ref{lemma4}}

This proof follows the same general outline as the proof of Lemma \ref{lemma2}, in the following sense:\\
First, we will show that for any subgraph $F\subseteq G$ with small maximum degree, and for any sparse, expanding subgraph $\Gamma \subseteq G$, such that large vertex subsets in $\Gamma$ have linear sized matching between them, the edge set $E(G\setminus F)$ contains a large set of booster pairs with respect to $\Gamma$, such that every edge does not participate in many pairs.\\
Then we construct an auxiliary graph $X$, such that every edge in $X$ represent a booster pair, and show that for $p_1 \leq p \leq p_2$, the probability that $X^p$ does not contain any edge is $\exp \left( -\Omega \left( \beta n^2p \right) \right)$.\\
Finally, using the union bound, summing over all possible subgraphs $F, \Gamma$, we prove the lemma.\\
Since most of the proof should be clear, given the proof of Lemma \ref{lemma2}, we will only prove the existence of a ``good" set of booster pairs, and state the lemma regarding the existence of an edge in $X^p$.

\begin{lemma} \label{lemmaBPs4}

Let $\Gamma \subseteq G$ be a non--Hamiltonian subgraph with at most $8\beta ^2 n\log n$ edges, such that 
\begin{itemize}
	\item $\Gamma$ is connected;
	\item $\Gamma$ is an $\left( \frac{\alpha}{3}\cdot n ,2\right)$--expander;
	\item For every two disjoint subsets $U,W\subseteq V(G)$ of size $|U|,|W|\geq \alpha n$, there is a matching in $\Gamma$ between $U$ and $W$ of size at least $\frac{\alpha ^3}{100}\beta n$;
\end{itemize}
and let $F\subseteq G$ with $\Delta (F) \leq 2k-2$. Then $G\setminus F$ contains a set $B\subseteq \binom{E(G)}{2}$ of booster pairs with respect to $\Gamma$, such that
\begin{itemize}
	\item $|B| \geq \frac{\alpha ^5}{8000}\beta \cdot n^3$;
	\item Every edge of $E(G\setminus F)$ is a member of at most $n$ booster pairs in $B$.
\end{itemize}
\end{lemma}

\begin{proof}
Let $P$ be a longest path in $\Gamma$, let $S$ be a set of $\frac{\alpha}{3}n$ starting vertices of paths obtained from $P$ by a sequence of rotations (such a set exists, by \cite{POS}), and for $s\in S$ let $T_s$ be a set of $\frac{\alpha}{3}n$ end vertices of paths obtained from $P_s$ by a sequence of rotations, $P_s$ being a path with starting vertex $s$ obtained from $P$ by rotations.\\
For each pair of vertices $s\in S,\ t\in T_s$ let $P_{s,t}$ be a path between $s$ and $t$ obtained from $P$ by rotations. We will add the following set of \emph{BP}'s to $B$:
\begin{enumerate}
	\item If $|N_G(s)\cap P|,|N_G(t)\cap P|\geq (2\alpha +\frac{1}{2}\beta )n$:\\
	Let $c\in P$ be a vertex such that, without loss of generality, at least $\alpha n + 1$ of $s$'s neighbours on $P$ precede $c$ along $P_{s,t}$, and at least $\alpha n +1$ of $t$'s neighbours on $P$ succeed $c$ (if no such $c$ exists, switch between $t$ and $s$). Let $L := \lbrace p(x)|x\in N_G(s) \cap P ,x \mbox{ precedes } c \rbrace$ and let $R := \lbrace s(x)|x\in N_G(t) \cap P ,x \mbox{ succeeds } c \rbrace$. By the definition of $c$, $|L|,|R| \geq \alpha n$ and the sets are disjoint, so there is a matching in $\Gamma$ between $L$ and $R$ of size $\frac{\alpha ^3}{100}\beta n$.\\
	For each edge $(l,r)$ in the matching, add to $B$ the pair $\lbrace (s,s(l)),(p(r),t) \rbrace$.
	\item If (without loss of generality) $|N_G(s)\cap P| < (2\alpha +\frac{1}{2}\beta )n$:\\
	In this case, $s$ has at least $\frac{1}{2}\beta n$ neighbours outside $P$. Let $e_1,...,e_{\frac{\alpha ^3}{100}\beta n}$ be some subset of them and let $ f_1,...,f_{\frac{\alpha ^3}{100}\beta n}$ be some subset of edges in $E(G)$ touching $t$.\\
	For every $i=1,...,\frac{\alpha ^3}{100}\beta n$, add the pair $\{ e_i,f_i \}$ to $B$.
\end{enumerate}
In each case we added $\frac{\alpha ^3}{100}\beta n$ \emph{BP}'s to $B$.  As in Lemma \ref{lammaBPs}, every pair was examined at most eight times, so in total at least $\frac{1}{8}\sum _{s\in S} |T_s|\cdot \frac{\alpha ^3}{100}\beta n = \frac{\alpha^5}{7200}\beta \cdot n^3$ distinct pairs were added.\\
In addition, every edge incident to $s\in S$ was included in at most one pair for every $t\in T_s$ for when $s$ was considered, and the same is true for every edge incident to $t\in T_s$ (for some $s \in S$), so overall, every edge is a member of at most $n$ pairs in $B$.\\
Finally, since each edge is a member of at most $n$ edge pairs in $B$, removing all pairs containing an edge of $E(\Gamma) \cup E(F)$ yields a set $B$ of \emph{BP}'s of size at least $$\frac{\alpha^5}{7200}\beta \cdot n^3 - n \cdot |E(\Gamma )\cup E(F)| = \frac{\alpha^5}{7200}\beta \cdot n^3 - O \left( n^2 \log n \right) \geq \frac{\alpha ^5}{8000}\beta \cdot n^3 .$$
\end{proof}

\noindent As in the proof of Lemma \ref{lemma2}, for $F,\Gamma \subseteq G$ we construct the auxiliary graph $X = (V,E)$, with vertex set $V= E(G)$ and edge set $E = \lbrace (e,f) \in \binom{V}{2}:\{e,f\} \in B \rbrace$, with $B$ being a set of booster pairs whose existence is guaranteed by Lemma \ref{lemmaBPs4}.

\begin{lemma}
Let $p_1 \leq p \leq p_2$. The probability that $X^p$ contains no edge is $\exp \left( -\Omega \left( \beta n^2p \right) \right)$.
\end{lemma}

\noindent The proof of this lemma is essentially identical to the proof in Lemma \ref{lemma2}, as is the proof of the following corollary, and from the corollary the derivation of Lemma \ref{lemma4}:

\begin{corol}
Let $p_1 \leq p \leq p_2$. Then the following holds with probability $1-n^{-\omega (1)}$:\\
For every subgraph $\Gamma$ of $G_p$ with at most $8\beta ^2 n\log n$ edges, such that 
\begin{itemize}
	\item $\Gamma$ is connected;
	\item $\Gamma$ is an $\left( \frac{\alpha}{3}\cdot n ,2\right)$--expander;
	\item For every two disjoint subsets $U,W\subseteq V(G)$ of size $|U|,|W|\geq \alpha n$, there is a matching in $\Gamma$ between $U$ and $W$ of size at least $\frac{\alpha}{100}\beta n$;
\end{itemize}
and for every subgraph $F\subseteq G_p$ such that $\Delta (F) \leq 2k-2$, $G_p\setminus F$ contains a booster pair with respect to $\Gamma$.
\end{corol}

\section{$(n,d,\lambda )$--graphs} \label{sec-ndl}

In this section we provide a proof of Theorem \ref{thmNDL}. Throughout the section, $C=10^8$, $c=1/400$, and $G$ is some fixed graph assumed to be an $(n,d,\lambda )$--graph with $d\geq \frac{C\cdot n\cdot \log \log n}{\log n}$ and $\lambda \leq \frac{c\cdot d^2}{n}$.

Let $$d_0 := 10^{-6}\cdot \log n. $$ As in the previous proofs, we prove the theorem by proving two main lemmas:

\begin{lemma} \label{lemma5}
With high probability for every $F\subseteq G_{\tau _{2k}}$ with $\Delta (F) \leq 2k-2$, $G_{\tau _{2k}} \setminus F$ contains a subgraph $\Gamma _0$ with at most $2d_0 n$ edges, which is an $\left( \frac{n}{4} ,2\right)$--expander such that for every two disjoint subsets $U,W\subseteq V(G)$ of size $|U|,|W|\geq \frac{1}{3}d$ it holds that $\nu _{\Gamma _0} (U,W) \geq \frac{1}{6}d$.
\end{lemma}

\begin{lemma} \label{lemma6}
With high probability for every $F\subseteq G_{\tau _{2k}}$ with $\Delta (F) \leq 2k-2$ and for every subgraph $\Gamma $ of $G_{\tau _{2k}} \setminus F$ with at most $3d_0 n\log n$ edges, which is an $\left( \frac{n}{4} ,2\right)$--expander such that for every two disjoint subsets $U,W\subseteq V(\Gamma )$ of size $|U|,|W|= \frac{1}{3}d$, there is a matching in $\Gamma _0$ between $U$ and $W$ of size at least $\frac{1}{6}d$, the graph $G_{\tau _{2k}}\setminus F$ contains a booster pair with respect to $\Gamma$. 
\end{lemma}

\noindent Theorem \ref{thmNDL} is obtained from Lemma \ref{lemma5} and Lemma \ref{lemma6} in the same way Theorem \ref{thmDirac} and Theorem \ref{thmJumbled} were obtained from the corresponding lemmas in previous sections.

Throughout the proofs we will assume that $d = o(n)$, since the complementing case when $d \geq a\cdot n$ is already covered by Theorem \ref{thmJumbled} or Theorem \ref{thmDirac}. Indeed, in that case one can observe that if $a<\frac{2}{3}$ then by Lemma \ref{lemmaMixing} $G$ is a $\left( \frac{1}{3}a, \frac{1}{10}a^3 \right)$--dense graph, and otherwise the graph is obviously super--Dirac.

\subsection{Proof of Lemma \ref{lemma5}}
For a graph $\Gamma$ on $n$ vertices, let $$SMALL(\Gamma ):=\lbrace v\in V(\Gamma ): d_{\Gamma}(v)\leq d_0 \rbrace.$$ We will first show that WHP $G_{\tau _{2k}}$ has the following properties:
\begin{itemize}
	\item[(R1)] $\Delta \left( G_{\tau _{2k}} \right) \leq 10 \log n$;
	\item[(R2)] $\left| SMALL \left(G_{\tau _{2k}} \right) \right| \leq n^{0.1}$;
	\item[(R3)] $\forall u,v\in SMALL \left(G_{\tau _{2k}} \right) : dist_{G_{\tau _{2k}}}(u,v) > 4$;
	\item[(R4)] $\forall U\subset V(G)\ s.t.\ 0.8d_0\leq |U| \leq \frac{5d^2}{n}:e_{G_{\tau _{2k}}}(U)<0.4d_0 \cdot |U|$;
	\item[(R5)] $\forall U,W\subset V(G)\ disjoint\ s.t.\ \frac{d^2}{n} \leq |U| \leq \frac{n}{4},|W|=2|U|:e_{G_{\tau _{2k}}}(U,V(G)\setminus (U\cup W)) \geq 0.1 |U|\log n$;
	\item[(R6)] $\forall U,W\subset V(G)\ disjoint\ s.t.\ |U|,|W| \geq \frac{1}{6}d : e_{G_{\tau _{2k}} }\left( U, W \right) \geq \frac{d^2\log n}{50n}$.
\end{itemize}

\begin{proof}
For each property, we will bound the probability of $G_{\tau _{2k}}$ failing to have it separately.\\
Let $N = |E(G)| = \frac{nd}{2}$, and let
$$p_1 = \frac{\log n}{d},\ p_2 = \frac{1.1\log n}{d},\ m_1 = N\cdot p_1 = \frac{1}{2}n\log n,\ m_2 = N\cdot p_2 = 0.55n\log n.$$

\begin{lemma} \label{lemmaTAU3}
With high probability $m_1 \leq \tau _{2k} \leq m_2$.
\end{lemma}
\begin{proof}
It suffices to show that
\begin{enumerate}
	\item $Pr\left[ \delta \left( G_{p_1} \right) \geq 2k \right] = o(1)$;
	\item $Pr\left[ \delta \left( G_{p_2} \right) < 2k \right] = o(1)$.
\end{enumerate}
We now prove both estimates.

\begin{enumerate}
	\item We will bound the probability using Chebyshev's inequality. Let $X$ be the random variable counting the number of vertices in $V(G)$ with degree less than $2k$ in $G_{p_1}$. First, bound $\mathbb{E}\left[ X \right]$ from below:
	\begin{eqnarray*}
	\mathbb{E}\left[ X \right] & = & \sum\limits_{v\in V(G)} Pr\left[ d_{G_{p_1}}(v) < 2k \right] = n\cdot \sum\limits_{i=0}^{2k-1} \binom{d}{i} {p_1}^i (1-p_1)^{d-i} \\
	& \geq & n\cdot \left( \frac{dp_1}{2k} \right) ^{2k-1} \cdot (1-p_1)^d \geq \omega (1) \cdot \exp \left( \log n - \frac{dp_1}{1-p_1} \right) = \omega (1).
	\end{eqnarray*}
	
	Next, we bound $Var[X]$ from above. By following the same first steps of the bound in the proof of Lemma \ref{lemmaTAUjumb}, we get
	\begin{eqnarray*}
	Var[X] & \leq & \mathbb{E}[X] + \sum _{(u, v) \in E(G)} p_1\cdot \sum _{i=0}^{2k-2} \sum _{j=0}^{2k-2} \binom{d}{i} \binom{d}{j} {p_1}^{i+j} (1-p_1) ^{2d -i -j -2} \\
	& \leq & \mathbb{E}[X] + \frac{1}{2}n\log n \cdot 4k^2 \left( \frac{e\log n}{(2k-2)(1-p_1)} \right) ^{4k-4} (1-p_1) ^{2d -2} \\
	& = & \mathbb{E}[X] + O\left( n\log ^{4k-3} n \cdot \exp (-2dp_1) \right) \\
	& = & \mathbb{E}[X] + o(1).
	\end{eqnarray*}
	
	Now, $\frac{Var[X]}{\mathbb{E}[X]^2} = O\left( \mathbb{E}[X] ^{-1} \right) = o (1),$ and by Chebyshev's inequality we get $$Pr\left[ \delta \left( G_{p_1} \right) \geq 2k \right] = Pr[X=0] = o(1).$$

	\item Let $v\in V(G)$. The probability that $d_{G_{p_2}}(v) <2k$ is at most
	\begin{eqnarray*}
	Pr\left[ d_{G_{p_2}}(v) <2k \right] & \leq & \sum\limits_{i=0}^{2k-1} \binom{d}{i}{p_2}^i(1-p_2)^{d-i} \leq \sum\limits_{i=0}^{2k-1} (4\log n)^i \exp \left( - 1.1 \log n \right) \leq n^{-1.05} .
	\end{eqnarray*}
	By the union bound we get
	$$Pr\left[ \delta \left( G_{p_2} \right) < 2k \right] \leq \sum _{v\in V(G)} Pr\left[ d_{G_{p_2}}(v) <2k \right] \leq n^{-0.05} = o(1).$$
\end{enumerate}

\end{proof}

We continue our proof. With Lemma \ref{lemmaTAU3} proved, it suffices to show that $Pr\left( G_{p_1} \notin (R) \right) = o(1),$ for an increasing property $(R)$, and that $Pr\left( G_{p_2} \notin (R) \right) = o(1)$ for a decreasing property $R$.

As the calculations involved in bounding the probability that $G_{2k}$ does not have properties $(R1)-(R4)$ are almost identical to the calculations in the proof of Lemma \ref{lemma1} involving their counterpart properties $(P1)-(P4)$, we omit these calculations to avoid repetition.

\begin{itemize}
	
	\item[(R5)] For sets $U,W\subseteq V(G)$ with $\frac{d^2}{n} \leq |U| \leq \frac{n}{4}$ and $|W|=2|U|$ we have $$\frac{1}{4}n\cdot |U| \leq |U|\cdot |V(G) \setminus (U \cup W)| \leq n\cdot |U|.$$
	By the fact that $|U|\geq \frac{d^2}{n}$ and $\lambda \leq \frac{cd^2}{n}$, we get $$ \lambda \sqrt{|U|\cdot n} \leq \frac{cd^2}{n} \sqrt{|U|\cdot n} = c\cdot d\cdot \sqrt{|U|d^2/n} \le c\cdot d|U| ,$$
	and by Lemma \ref{lemmaMixing}, $$
	e_G(U,V(G)\setminus (U\cup W)) \geq \frac{1}{4}d\cdot |U| - \lambda \sqrt{|U|\cdot n} \geq 0.2d|U|.$$
	
	So
		\begin{eqnarray*}
		Pr\left[ G_{p_1} \notin {\bf (R5)} \right] & \leq & \sum\limits_{i=\frac{d^2}{n}}^{\frac{n}{4}} \binom{n}{i} \binom{n}{2i} Pr\left[ Bin\left( 0.2d\cdot i , p_1 \right) \leq 0.1 i \log n \right] \\
		& \leq & \sum\limits_{i=\frac{d^2}{n}}^{\frac{n}{4}} \left( \frac{en}{2i} \right) ^{2i} \cdot \exp \left( -\Omega \left( i\log n \right) \right) = o(1) ,
	\end{eqnarray*}
	with the last bound obtained by observing that for $i \geq \frac{d^2}{n} =  n^{1-o(1)}$ we have $\log \left( \frac{n}{i} \right) = o(\log (n))$.
	
	\item[(R6)] By Lemma \ref{lemmaMixing}, if $|U|,|W| \geq \frac{1}{6}d$ and $\lambda \leq \frac{cd^2}{n}$ we have $$e_G(U,W) \geq \frac{d}{n}|U||W| - \lambda \sqrt{|U||W|} \geq \frac{d^3}{40n}.$$ So
	\begin{eqnarray*}
		Pr\left[ G_{p_1} \notin {\bf (R6)} \right] & \leq & \binom{n}{\frac{1}{6}d}^2 \cdot Pr\left[ Bin\left( \frac{d^3}{40n} , p_1 \right) < \frac{d ^2\log n}{50n} \right] \\
		& \leq & \left( \frac{6en}{d} \right) ^{\frac{1}{3}d} \cdot \exp \left( - \frac{d^2\log n}{40\cdot 5^2\cdot 2n} \right) \\
		& \leq & \exp \left( -0.5C \cdot d\log \log n \right) = o(1) ,
	\end{eqnarray*}

	here using the assumption that $d \geq C\cdot \frac{n\log \log n}{\log n}$.
	
\end{itemize}
\end{proof}

\begin{lemma} \label{lemma5.1}
Let $\Gamma$ be a graph on $n$ vertices such that $\delta \left( \Gamma \right) \geq 2k$ and such that $\Gamma \in (R1)$--$(R5)$, and let $F\subseteq \Gamma$ with $\Delta (F) \leq 2k-2$. Then $\Gamma \setminus F$ contains a subgraph $\Gamma ^{(1)}$ which is an $\left( \frac{n}{4},2 \right)$--expander with at most $d_0n$ edges.
\end{lemma}

\begin{proof}
The proof is essentially identical to the proof of Lemma \ref{lemma1.1}, up to some constants being different.
\end{proof}

\begin{lemma} \label{lemma5.2}
Let $\Gamma$ be a graph on $n$ vertices such that $\delta \left( \Gamma \right) \geq 2k$, $|E(\Gamma )| = m_1$ and such that $\Gamma \in (R1),(R6)$, and let $F\subseteq \Gamma$ with $\Delta (F) \leq 2k-2$. Then $\Gamma \setminus F$ contains a subgraph $\Gamma ^{(2)}$ which has at most $d_0n$ edges, such that for every two disjoint subsets $U,W\subseteq V(\Gamma )$ of size $|U|,|W|= \frac{1}{3} d$ it holds that $\nu _{\Gamma ^{(2)}} (U,W) \geq \frac{1}{6}d$.
\end{lemma}

\begin{proof}
Let $\Gamma ^{\prime} := \Gamma \setminus F$, $\rho = 10^{-6}$, and let $\Gamma ^{(2)}$ be a random subgraph of $\Gamma ^{\prime}$ distributed according to $\Gamma ^{\prime}_{\rho}$. We will show that with positive probability $\Gamma ^{(2)}$ satisfies the desired properties, and therefore prove the existence of the desired subgraph.

Let $U,W\subseteq V(G)$ be some vertex sets of size $\frac{1}{3}d$. If there is no matching of size $\frac{1}{6}d$ between $U$ and $W$ in $\Gamma ^{(2)}$ then there are two subsets, $U^{\prime}\subseteq U, W^{\prime}\subseteq W$ of size $\frac{1}{6}d$, such that $E_{\Gamma ^{(2)}}\left( U^{\prime},W^{\prime} \right) =\emptyset$. By the union bound and by the fact that $\Gamma \in (R6)$, the probability that such a pair of subsets exists is at most $$\binom{n}{\frac{1}{6}d}^2\cdot  (1-\rho )^{\frac{d^2\log n}{50n} - \frac{1}{3}d\cdot(2k-2)} \leq \exp \left( \left( 0.4-\rho\cdot \frac{C}{50} \right) \cdot d\log \log n +O(d) \right) = o(1).$$
Observing that the probability that $\Gamma ^{(2)}$ has at most $d_0 n$ edges is of order $\Omega (1)$, we conclude that the probability that $\Gamma ^{(2)}$ satisfies both conditions is positive, and so such a subgraph exists.
\end{proof}

We obtain Lemma \ref{lemma5} by setting $\Gamma _0 = \Gamma ^{(1)} \cup \Gamma ^{(2)}$.

\subsection{Proof of Lemma \ref{lemma6}}
This proof follows the same general outline as the proofs of Lemma \ref{lemma2} and of Lemma \ref{lemma4}.\\
Similarly to the previous proofs, we will first show that if $F\subseteq G$ is a subgraph with $\Delta (F) \leq 2k-2$, and $\Gamma \subseteq G$ is a sparse, expanding subgraph with the property of having an $\Omega (d)$-sized matching between large vertex subsets, then $G\setminus F$ contains a set of $\Omega (n^2d)$ booster pairs with respect to $\Gamma$ in which the number of booster pairs containing each edge is bounded from above by an amount linear in $n$.

We will then, with the aid of an auxiliary graph, and by applying the union bound over all sparse subgraphs, show that with high probability $G_{\tau _{2k}}$ contains a booster pair for all the expanding subgraphs contained in $G_{\tau _{2k}}$ with these properties.

Since this proof bears many similarities to the relevant previous proofs in this paper, we only sketch it briefly.

\begin{lemma} \label{lemmaBPs6}

Let $\Gamma \subseteq G$ be a non--Hamiltonian subgraph with at most $3d_0n$ edges, such that $\Gamma$ is an $\left( \frac{n}{4} ,2\right)$--expander, and such that for every two disjoint subsets $U,W\subseteq V(G)$ of size $|U|,|W| = \frac{1}{3}d$, there is a matching in $\Gamma$ between $U$ and $W$ of size at least $\frac{1}{6}d$, and let $F\subseteq G$ with $\Delta (F) \leq 2k-2$. then $G\setminus F$ contains a set $B\subseteq \binom{E(G)}{2}$ of booster pairs with respect to $\Gamma$, such that
\begin{itemize}
	\item $|B| \geq \frac{1}{800}n^2d$;
	\item Every edge of $E(G\setminus F)$ is a member of at most $\frac{n}{2}$ booster pairs in $B$.
\end{itemize}
\end{lemma}

\begin{proof}
The proof is essentially identical to the proof of Lemma \ref{lemmaBPs4}, up to some minor changes.
\end{proof}

As in the previous proofs, for $F,\Gamma \subseteq G$ we construct the auxiliary graph $X = (V,E)$, with vertex set $V= E(G)$ and edge set $E = \lbrace (e,f) \in \binom{V}{2}:\{e,f\} \in B \rbrace$, with $B$ being a set of booster pairs whose existence is guaranteed by Lemma \ref{lemmaBPs6}.

\begin{lemma}
Let $p_1 \leq p \leq p_2$. The probability that $X^p$ contains no edge is at most $2 \exp \left( -\frac{1}{8100}ndp \right)$.
\end{lemma}

The proof of this lemma is essentially identical to the proof in Lemma \ref{lemma2}, as is the proof of the following corollary, and from the corollary the derivation of Lemma \ref{lemma6}:

\begin{corol}
Let $p_1 \leq p \leq p_2$. Then the following holds with probability $1-n^{-\omega (1)}$:\\
For every subgraph $\Gamma$ of $G_p$ with at most $3d_0n$ edges, such that $\Gamma$ is an $\left( \frac{n}{4} ,2\right)$--expander and such that for every pair of disjoint subsets $U,W\subseteq V(G)$ with $|U|,|W| = \frac{1}{3}d$: $\nu _{\Gamma}(U,W) \geq \frac{1}{6}d$, and for every subgraph $F\subseteq G_p$ such that $\Delta (F) \leq 2k-2$, $G_p\setminus F$ contains a booster pair with respect to $\Gamma$.
\end{corol}

\section{Concluding remarks} \label{sec-remarks}
It is worth noting that the proofs of Theorem \ref{thmDirac}, Theorem \ref{thmJumbled} and Theorem \ref{thmNDL} all work if we allow the required number of edge disjoint cycles $k$ to grow mildly with $n$, with no changes to the proofs needed. More specifically, for Theorem \ref{thmDirac} and Theorem \ref{thmJumbled} it is sufficient to assume that $k=k(n) =o(\log n)$, and in Theorem \ref{thmNDL} the assumption $k=k(n) =o(\log \log n)$ suffices.

Johansson \cite{RSATALK} provided the following example of a Dirac graph on $n$ vertices: $V(G) = A \cup B$ s.t. $|A|=|B| = \frac{n}{2}$ and $E(G) = \binom{B}{2} \cup (A \times B)$. This is a Dirac graph for which $G_{\tau _2}$ is not Hamiltonian with probability bounded away from 0, thus showing that the assumption that $G$ is a Dirac graph is not sufficient for a hitting time result. In our statement of Theorem \ref{thmDirac} we assume that the base graph $G$ satisfies $\delta (G) - \frac{1}{2}n = \Omega (n)$, and this assumption is necessary for our proof of the theorem. We leave it as an open question whether a similar hitting time result can be proven under a milder restriction on the difference $\delta (G) - \frac{1}{2}n$.

Finally, our result in Theorem \ref{thmNDL}, along with the result by Frieze and Krivelevich in \cite{FK2}, provide hitting time statements for $(n,d,\lambda )$--graphs with $d = \omega \left( (n\log n)^{3/4} \right)$. We leave it as an open question whether this range can be extended, and whether some of the restrictions on $\lambda$ can be eased in this range.

\end{document}